\documentclass[prodmode,acmtomacs]{acmsmall} 
\usepackage{amsfonts,amsmath,amssymb,tikz,textcomp, multirow, fp}
\usetikzlibrary{arrows,automata}
\usepackage[english]{babel}
\usepackage{subfig}
\usepackage{algorithm,algpseudocode,algorithmicx}
\usepackage{ifthen, wrapfig, setspace}

\newcommand{\rda}[1]{\textcolor{black}{#1}}
\newcommand{\rdb}[1]{\textcolor{black}{#1}}
\newcommand{\rdc}[1]{\textcolor{black}{#1}}
\newcommand{\rdd}[1]{\textcolor{black}{#1}}
\newcommand{\rde}[1]{\textcolor{black}{#1}}
\newcommand{\rdf}[1]{\textcolor{black}{#1}}

\newcommand{\E}{{\mathbb{E}}}
\newcommand{\ssX}{{\cal{X}}}

\newcommand{\N}{{\mathbb{N}}}

\renewcommand{\P}{\mathbb{P}}

\newcommand{\dd}{d}

\newcommand{\PPP}{\bar{\mathbb{P}}}
\newcommand{\ppp}{\bar{p}}
\newcommand{\ddd}{\bar{d}}
\newcommand{\www}{\w^{\dom}}

\newcommand{\Q}{\mathbb{Q}}
\newcommand{\df}{\triangleq}

\newcommand{\asyrate}{r}
\renewcommand{\i}{{\bf 1}}
\newcommand{\unaiv}{{\bf v}}

\newcommand{\pth}{\omega}
\newcommand{\pths}{\Omega}
\newcommand{\ipths}{\Phi}
\newcommand{\state}{x}
\newcommand{\statea}{\state}
\newcommand{\stateb}{z}

\newcommand{\init}{s}

\newcommand{\intprobest}{\hat{\ipf}}
\newcommand{\ipprobest}{\hat\ipf}

\newcommand{\ipf}{\pi}

\newcommand{\nsamp}{N}

\newcommand{\ex}[1]{$\cdot$10$^{#1}$}

\newcommand{\ps}[1]{\pth(#1)}
\newcommand{\pthl}{n_{\pth}}

\newcommand{\gstate}{g}
\newcommand{\tstate}{t}
\newcommand{\core}{\Lambda}

\newcommand{\edge}{\Gamma}
\newcommand{\stopis}{m}

\newcommand{\w}{v}

\newcommand{\dom}{\Delta}
\newcommand{\ndom}{\Psi}
\newcommand{\tot}{\ipths}
\DeclareMathOperator{\vr}{Var}
\newcommand{\vrq}{\text{Var}_\Q}

\renewcommand{\dim}{D}

\newcommand{\pl}{+}
\newcommand{\plpl}{++}

\newcommand{\website}{\rdd{\tt http://datashare.is.ed.ac.uk/handle/10283/2630}}

\newcommand{\citet}{\citeN}

\begin{document}

\markboth{D. Reijsbergen et al.}{Path-ZVA: general, efficient and automated importance sampling for HRMSs (preprint)}

\title{Path-ZVA: general, efficient and automated importance sampling for highly reliable Markovian systems (preprint)}
\author{DANI\"EL REIJSBERGEN
\affil{University of Edinburgh, Scotland}
PIETER-TJERK DE BOER
\affil{University of Twente, Netherlands}
WERNER SCHEINHARDT
\affil{University of Twente, Netherlands}
SANDEEP JUNEJA
\affil{Tata Institute of Fundamental Research, India}}

\begin{abstract}
We introduce Path-ZVA: an efficient simulation technique for estimating the probability of reaching a rare goal state before a \rdd{regeneration} state in a (discrete-time) Markov chain. Standard Monte Carlo simulation techniques do not work well for rare events, so we use importance sampling; i.e., we change the probability measure governing the Markov chain such that transitions `towards' the goal state become more likely. To do this we need an idea of distance to the goal state, so some level of knowledge of the Markov chain is required. In this paper, we use graph analysis to obtain this knowledge. In particular, we focus on knowledge of the shortest paths (in terms of `rare' transitions) to the goal state. 
\rde{We show that only a subset of the
(possibly huge) state space needs to be considered.
\rdf{This is effective when the high dependability of the system is primarily due to high component reliability,
but less so when it is due to high redundancies}.
For several models we compare our results to well-known
importance sampling methods from the literature and demonstrate the large potential gains of our method.}

\end{abstract}




\acmformat{(...)}

\begin{bottomstuff}
%
\end{bottomstuff}

\maketitle

\definecolor{green2}{rgb}{0,0.6,0}
\definecolor{red2}{rgb}{0.8,0,0}
\newcommand{\red}{\color{red2}}

\newcommand{\ctlpath}{\phi}
\newcommand{\likrat}{L_\Q}

\section{Introduction} \label{sec: introduction}

Critical systems and infrastructures are increasingly required to be highly reliable,
which has implications not only for the reliability of individual system components,
but also for the accuracy of model-based evaluation.
Realistic models of highly reliable systems typically have very large state spaces.
Additionally, low component failure rates or a wide range of included system behaviours mean
that a model may exhibit multiple \emph{time scales},
in which system failure is the unlikely result of low-intensity state transitions (e.g., component failures)
\rdd{taking precedence over} high-intensity transitions (e.g., component repairs).
Numerical methods for evaluating system failure probabilities such as those implemented in the model checking tool PRISM \cite{kwiatkowska2011prism} --- e.g., the Gauss-Seidel method --- typically prove to be computationally infeasible due to the size of state space.
Furthermore, state space reduction techniques that ignore low-intensity behaviour risk disposing of \rdd{unlikely but interesting} events.

A common and generally applicable alternative is \emph{Monte Carlo simulation}, which only requires an implicit description of the state space and is therefore largely independent of its size.
However, if the interesting behaviour is unlikely, a prohibitively large number of simulation runs is typically required before the rare event of interest is first observed.
Hence, there is a need for hybrid techniques that strike a compromise between numerical techniques and standard Monte Carlo, whilst maintaining, to the largest possible extent, the \rdd{general applicability} of both methods.

In this paper we focus on Markovian systems in which individual components can fail and be repaired,
and system failure occurs when certain combinations of components have failed.
Crucially, we assume that the component failure rates have a (much) smaller order of magnitude than the repair rates. This is formalised in the notion of \emph{highly reliable Markovian systems} (HRMSs), which include any Markov chain in which rates are parameterised by powers of some rarity parameter~$\epsilon$,
where higher powers of~$\epsilon$ correspond to component failures.
Our (very small) probability of interest is that of reaching a system failure state within
one \rdd{regeneration cycle}, i.e., between two visits to a given \rdd{regeneration} state.
Once this quantity has been estimated,
renewal theory \cite{cox1962renewal} can be used to
calculate many system performance measures of practical interest, such as the mean time to failure, the unreliability, and the unavailability,
without the need to estimate any other quantities that involve rare events.

Our starting point will be a discrete time Markov chain (DTMC) \rdd{with fixed state space size and structure}. When the HRMS is a Markov chain in continuous time (as is usually the case), we simply consider the DTMC embedded at transition times (replacing the transition rates by normalized transition probabilities). This is allowed since the probability of our interest does not depend on the times spent in states, and hence is the same in the embedded DTMC as in the original system.

To estimate rare event probabilities in the DTMC, we use \emph{importance sampling}
--- a simulation method in which transitions that lead to the rare event are made more likely \cite{heidelberger1995fast}.
More precisely, we follow a so-called Zero Variance Approximation (ZVA) scheme, based on some \emph{a priori} approximation of the probability of interest. For this approximation we use \emph{path-based} measures for the distance from each state to the target state, which is why we call our method ``Path-ZVA''. Our distance measure is the number of `failure' transitions needed to get to the target state, or, in general, the $\epsilon$-power of the most likely path to get there, and during the simulation we will `push' the system in a direction that minimises this distance.
It turns out that \rdf{in many cases} only a small part of the state space needs to be considered
to find the {\em relevant} paths, making the method computationally advantageous.
Hence, our method consists of \emph{(i) a pre-processing step}, in which
a graph-analysis algorithm finds the shortest paths on a subset of the state space,
followed by \emph{(ii) the actual simulations}, using an importance sampling scheme (based on the shortest paths) for efficiently estimating the probability of interest over the entire state space.

This Path-ZVA procedure is:\\
\hspace{1cm} \emph{(1) general}, as the \emph{only} requirements on the Markov model are that it is parameterised using $\epsilon$-orders and that the relevant subset \rdd{is numerically better tractable} than the state space \rdf{as a whole};\\
	\emph{(2) efficient}, as it provably has the desirable properties of either Bounded Relative Error and Bounded Normal Approximation, or Vanishing Relative Error \rdd{for small $\epsilon$}; and\\
	\emph{(3) automated}, as the algorithm requires no user input apart from the model description. The code is available on \website. 

The remainder of this paper is as follows. After a formal description of the model setting and of (ZVA) importance sampling simulation in Section~\ref{sec: preliminaries},
we describe our `Path-ZVA algorithm' in detail in Section~\ref{sec: dijkstra algorithm}.
Next, we prove in Section~\ref{sec: bre} that the resulting estimators have several desirable efficiency
properties.
We discuss a further variance reduction technique in Section~\ref{sec: vrff},
which makes optimal use of the pre-processing step. Finally we present an empirical evaluation of all the discussed techniques in Section~\ref{sec: results}, including a comprehensive case study involving several benchmark models from the literature.
Most of this paper is based on Chapters 5 and 6 of \cite{reijsbergen2013efficient}.

\section{Model \& Preliminaries} \label{sec: preliminaries}

\subsection{Model setting} \label{sec: model}

The model is given in terms of a discrete-time Markov chain (DTMC) with a (large, possibly infinite) state space $\ssX$. (Note that the timing behaviour of the system is not important to the method; in fact, the DTMCs of the multicomponent systems of Section~\ref{sec: results} are the underlying DTMCs of continuous-time Markov Chains.) We assume that the system starts in a unique initial state $\init \in \ssX$,
and that there is a single goal state $\gstate \in \ssX$
({potentially after merging all states from a bigger goal set into a single state}).
We also assume that there is (again after a potential merge) a single \rdd{taboo/regeneration} state $\tstate \in \ssX$.
Note that it may \emph{not} be clear \emph{a priori} which states are to be merged into $\gstate$ and $\tstate$; since large DTMCs are typically described using a high-level language (e.g., a stochastic Petri net), we determine which states in the relevant part of the state space to collapse into $\gstate$ and $\tstate$ on-the-fly whilst running the algorithm described in Section~\ref{sec: dijkstra algorithm}.
 In the following, we assume that $\init \neq \gstate$ and $\init \neq \tstate$.
We are no longer interested in the behaviour of the system once the system hits $\gstate$ or $\tstate$ so we assume that these states {are absorbing, i.e.,} have a self-loop with probability $1$.

The complete transition probability structure in the DTMC is given by the probabilit\rdb{ies} $p_{\statea \stateb}$ of jumping from state $\statea$ to state $\stateb$\rdb{, with $\statea, \stateb \in \ssX$}. Th{e} probabilities {$p_{\statea \stateb}$} depend on $\epsilon$, the time-scale parameter that formalises the notion that there are fundamental differences between the transition probabilities in the DTMC. To say more about the dependence on $\epsilon$ we will write in the sequel, with $f,h, f_y, g_y : \mathbb{R} \rightarrow \mathbb{R}$:
\begin{tabular}{ll}
$f(\rdb{\epsilon}) = \Theta(\rda{h}(\rdb{\epsilon}))$
& iff\quad $0 < \lim_{\rdb{\epsilon} \downarrow 0} {f(\rdb{\epsilon})}/{\rda{h}(\rdb{\epsilon})} < \infty,$\\
$f(\rdb{\epsilon}) = O(\rda{h}(\rdb{\epsilon}))$
& iff\quad $\lim_{\rdb{\epsilon} \downarrow 0} {f(\rdb{\epsilon})}/{\rda{h}(\rdb{\epsilon})} < \infty,$\\
$f(\rdb{\epsilon}) = o(\rda{h}(\rdb{\epsilon}))$ 
& iff\quad $\lim_{\rdb{\epsilon} \downarrow 0} {f(\rdb{\epsilon})}/{\rda{h}(\rdb{\epsilon})} = 0,$\\
$f_y(\epsilon) = \Theta(g_y(\epsilon))$ \emph{uniformly} in $y$
&iff\quad $\exists a,b>0$ such that $\forall y: a < \lim_{\epsilon\downarrow 0} {f_y(\epsilon)}/{g_y(\epsilon)} < b,$
\end{tabular}
assuming these limits exist.

Throughout, our assumption is that for all non-zero transition probabilities $p_{\state\stateb} > 0$, \rdb{some $\asyrate_{\statea \stateb} \in \rdd{\N \cup \{\infty\}}$ exists such that}
$p_{\statea \stateb}(\epsilon) = \Theta(\epsilon^{\asyrate_{\statea \stateb}}).$
If $p_{\state\stateb} = 0$, we set $\asyrate_{\statea \stateb}$ \rdd{equal to $\infty$.}
\rdd{Note} that $r_{\statea \stateb}$ for \mbox{$\statea, \stateb \in \ssX$} are fixed parameters of the model. An example of a DTMC parameterised in this way can be found in Figure~\ref{fig: model example}. (See Section~\ref{sec: practical application} for a discussion for how these $r_{\statea \stateb}$ are chosen in practice.)

\begin{figure}[t!]
\centering
	\caption{Example of a HRMS with initial state $\init$, goal state $\gstate$ and \rdd{regeneration} state $\tstate$.}
	\label{fig: model example}
	\vspace{-0.2cm}
	\begin{tikzpicture}[->,>=stealth',shorten >=1pt,auto,node distance=3cm,semithick,scale=0.8] 
		\tikzstyle{every state}=[fill=blue!20,draw=black,thick,text=black,scale=0.5]
		\node at (0,0) (s0) {};

		\node[state] at (0,-0.6) (s02) {\huge $\tstate$};
		
		\tikzstyle{every state}=[fill=white,draw=black,thick,text=black,scale=0.5]
		\node[state] at (0,0.6) (s01) {\huge $\init$};
		\node[state] (s1) [right of = s0] {\huge $1$};
		\node[state] (s2) [right of = s1] {\huge $2$};
		\node[state] (s3) [right of = s2] {\huge $3$};
		\node[state] (s4) [right of = s3] {\huge $4$};
		
		\tikzstyle{every state}=[fill=red!20,draw=black,thick,text=black,scale=0.5]
		\node[state] (sr) [right of = s4] {\huge $\gstate$};

			\path (s01) edge [bend left] node[above] {\footnotesize $1$} (s1);
			\path (s1) edge  [bend left] node[below] {\footnotesize \textcolor{black}{$1-\epsilon$}} (s02);
			\path (s1) edge  [bend left] node[above] {\footnotesize \textcolor{black}{$\epsilon$}} (s2);
			\path (s2) edge  [bend left] node[below] {\footnotesize $1-\epsilon$} (s1);
			\path (s2) edge  [bend left] node[above] {\footnotesize $\epsilon$} (s3);
			\path (s3) edge  [bend left] node[below] {\footnotesize $1-\epsilon$} (s2);
			\path (s3) edge  [bend left] node[above] {\footnotesize $\epsilon$} (s4);
			\path (s4) edge  [bend left] node[below] {\footnotesize $1-\epsilon$} (s3);
			\path (s4) edge  [] node[above] {\footnotesize $\epsilon$} (sr);
	\end{tikzpicture}
	\vspace{-0.4cm}
	\label{fig: mm1 queue}
\end{figure}
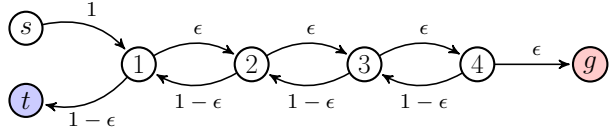
Let a \emph{path} $\pth$ be a sequence $\ps{0},\ps{1},\ldots,\ps{\pthl}$ of states in $\ssX$, with $\pthl$ denoting the number of steps in the path. Let $\pths(\state)$ be the set of paths $\pth$ \rdb{starting at} $\ps{0} = \state$. We are interested in the event that the system reaches $\gstate$ before $\tstate$. To formalise this, \rdb{let} $\forall \state \in \ssX$
\begin{equation*}
\ipths(\state) \df \left\{ \pth \in \pths(\state): \ps{\pthl} = \gstate, \,\,\, \forall k < \pthl : \ps{k} \notin \{\tstate, \gstate\}\right\}
\end{equation*}
\rdb{be} the set of all paths starting in $\state$ in which the event of interest occurs and which terminate as soon as $\gstate$ is reached. For all $\state \in \ssX$ we define the probability that the rare event occurs, starting in $\state$, as
\begin{equation}
\ipf(\state) \df\sum_{\pth \in \ipths(\state)} \P(\pth) = \P(\ipths(\state))  \qquad \mbox{ where }  \qquad \P(\pth) \df \prod_{i=1}^{\pthl} p_{\ps{i-1} \ps{i}}. \label{eq def mu}
\end{equation}
We are interested in~\mbox{$\ipf \df \ipf(\init)$} and estimate this probability using simulation, as described below.

\subsection{Simulation} \label{sec: simulation}

The basic means to evaluate $\ipf$ is through a point estimate; to obtain one, we draw~\mbox{$\nsamp \in \N$} sample paths to obtain a sample set $\{\pth_1,\ldots,\pth_{\nsamp}\}$. To draw a sample path, we start in $\init$ and draw successor states using $\P$ until we reach \rdb{either} $\tstate$ or $\gstate$. (We assume that this happens in finite time with probability \rdd{$1$.)}
Let $\ipths \df \ipths(\init)$, and $\i_{\ipths}(\pth)$ denote an indicator function which equals $1$ if $\pth$ is in $\ipths$ and $0$ otherwise; this allows us to obtain an unbiased estimator of $\ipf$, given by 
\begin{equation}
\hat{\ipf}_{\P} = \displaystyle\frac{1}{\nsamp} \sum_{\rdb{k}=1}^{\nsamp} \i_{\ipths}(\pth_{\rdb{k}}). \label{eq monte carlo} 
\end{equation}
An approximate 95\%-confidence interval for $\ipf$ can be obtained using the Central Limit Theorem (see, e.g., \citet[\S 4.5]{law1991simulation})

\begin{table}
	\tbl{List of symbols}{
		\begin{tabular}{c l}
			$\ssX$ & state space of the Markov chain \\
			$\init$, $\gstate$, $\tstate$ & initial state, goal state, and \rdd{taboo/regeneration} state respectively \\
			$\rdb{\P}$, $p_{\statea\stateb}$ & original probability of the transition from state $\statea$ to state $\stateb$ \\
			$\rdb{\Q}$, $q_{\statea\stateb}$ & new probability of the transition from state $\statea$ to state $\stateb$ \\
			$\rdd{\asyrate_{\statea\stateb}}$ & \rdd{$\epsilon$-order of the transition from state $\statea$ to state $\stateb$, i.e., $p_{\statea\stateb} = \Theta(\epsilon^{\asyrate_{\statea\stateb}})$ } \\
			$\pth$ & a path, i.e., a sequence of states $\ps{0},\ps{1},\ldots,\ps{\pthl}$  \\
			$\P(\cdot)$ & probability of a \rdb{(set of) path(s) under $p_{\statea\stateb}$}\\
			$\Q(\cdot)$ & \rda{probability of a \rdb{(set of) path(s)} under $q_{\statea\stateb}$}\\
			$\pths(\state)$ & set of all paths $\pth$ \rdb{starting at} $\ps{0} = \state$ \\
			$\ipths(\state)$ & set of all \rdb{`successful'} path\rdb{s i}n $\pths(\state)$, in which $\gstate$ is reached before $\tstate$ \\  %
			$\ipf(\state)$ & $\sum_{\pth \in \ipths(x)} \P(\pth) = \P(\ipths(\state))$ \\
			$d(\state,\stateb)$ & shortest distance from $\state$ to $\stateb$ in terms of $\epsilon$-orders \\
			$\core$ & all states $\state$ for which $d(\init,\state) \leq d(\init,\gstate)$, \rdb{including $\init$ and $\gstate$} \\
			$\edge$ & all states $\stateb$ in $\ssX \setminus \core$ such that $\exists \statea \in \core$ s.t. $p_{\statea\stateb} > 0$ \\
			$\ppp_{\statea\stateb}$, $\PPP$, $\ddd$ & $p_{\statea\stateb}$, $\P$, $d$ as before, but in the system in which states in $\edge$ \\
			& \;\; have been given a transition with probability $1$ to $\gstate$ \\
			$\dom(x)$ & \rdb{`dominant'} paths $\pth \in \ipths(\state)$, for which $\rdb{\PPP}(\pth) = \Theta(\epsilon^{\rdd{\ddd}(\state,\gstate)})$ \\
			$\ipths$, $\ipf$, $\dom$ & shorthand notation for $\ipths(\init)$, $\ipf(\init)$, $\dom(\init)$ \\
			$\www(\state)$ & $\sum_{\pth \in \dom(x)} \PPP(\pth) = \PPP(\dom(\state))$, \rdb{approximation of $\ipf(\state)$} \\
	\end{tabular}}
\end{table}

As we discussed in the introduction, we use importance sampling: we simulate using different transition probabilities $(q_{\statea \stateb})_{\statea, \stateb \in \ssX}$ under which paths in $\ipths$ are more likely. \rda{Let $\Q$ be the probability measure on paths defined analogously to $\P$ but for $q_{\statea\stateb}$.} We compensate for overestimation by weighting {each outcome} with the ratio of $\P$ and $\Q$.
Every time a transition is sampled using the new probabilities, this weighting factor needs to be incorporated.
Our new estimator --- replacing \eqref{eq monte carlo} --- then becomes
 \begin{equation}
  \hat{\ipf}_{\Q} = \displaystyle\frac{1}{N} \sum_{\rdb{k}=1}^N \likrat(\omega_{\rdb{k}}) \cdot \i_{\ipths}(\pth_{\rdb{k}}),
 \label{eq importance sampling estimator}
 \quad\text{with}\quad
  \likrat(\pth) = \prod_{i=1}^{\pthl}  \frac{{p_{\ps{i-1} \ps{i}}}}{{q_{\ps{i-1} \ps{i}}}}.
 \end{equation}
This estimator is unbiased for any new distribution that assigns positive probability to transitions that have positive probability under the old distribution on paths in $\ipths(\init)$ {(by the Radon-Nikodym Theorem, see Chapter~7 of Capi\'nski and Kopp [2004])}. \nocite{capinski2004measure} In the following, we will write $\hat{\ipf}  = \hat{\ipf}_{\Q} $ for brevity.

If $\Q$ is chosen carefully, the estimator based on \eqref{eq importance sampling estimator} will have a lower variance than the standard estimator. The performance of an importance sampling method is measured by the variance of $\hat{\ipf}$ under $\Q$, given by
\begin{equation*}
\vr_{\Q}(\intprobest) = \E_\Q \left( \likrat^2 \cdot {\bf 1}_{\ipths}\right)- {\ipf}^2.
\end{equation*}
Using $\Q = \P$, we obtain the variance of the standard estimator: $\ipf(1 - \ipf)$. 
A particularly interesting efficiency metric for an estimator is its \emph{relative error}, given by
\begin{equation*}
\frac{\sqrt{\vr_{\Q}(\intprobest) }}{\ipf}.
\end{equation*}
The relative error of the standard estimator is given by $\sqrt{(1-\ipf) / \ipf}$, which goes to infinity when $\ipf$ goes to zero. When the relative error of an estimator remains bounded when $\ipf$ goes to zero, we say that our estimator has the desirable property of \emph{Bounded Relative Error} (BRE). When it goes to zero, we say that it has the even more desirable property of \emph{Vanishing Relative Error} (VRE) \cite{lecuyer2010asymptotic}.

We use the Zero Variance Approximation (ZVA) approach (cf. \cite{lecuyer2008approximate}),
and present the following probability measure $\Q$
\begin{equation}\label{eq zva def}
 q_{\statea\stateb} \df
  \frac{ p_{\statea\stateb} \w(\stateb)}
        {\sum_{\statea' \in \ssX} p_{\statea \statea'} \w(\statea')}.
\end{equation}
where $v(z)$ is some {\em approximation} for the true probablity $\pi(z)$.
Clearly, if $v(z)$ were exactly equal to $\pi(z)$, the denominator would
be $\pi(x)$ and the estimator would have zero variance \cite{deboer2007estimating},
but of course we do not explicitly know $\pi(\cdot)$.
If the simulation distribution $\Q$ associated with the approximation $\w$ is good enough then we have succeeded in overcoming the main problem facing standard Monte Carlo simulation of rare events.
The particular ZVA technique (choice of $v$) discussed in this paper --- namely Path-ZVA --- will be the subject of Section~\ref{sec: dijkstra algorithm}.

\subsection{Related work} \label{sec: related work}

In this section we give a brief overview of papers on the use of importance sampling for Highly Reliable Markovian Systems that we consider to be particularly relevant to this paper, either because they discuss literature benchmark\rda{s} o\rda{r} because they discuss recent advances.
As a first remark, note that our notion of an HRMS
(namely any Markov chain in which the transitions are \rdd{given} $\epsilon$-orders)
is more general than what is typically considered in the literature.
In the literature, an HRMS is often restricted to what we call a \emph{multicomponent system},
where only failure transitions have rates of order $\epsilon$, while we do not have this restriction.

The first application of an importance sampling method --- namely \emph{failure biasing} --- to HRMSs goes back to \citet{lewis1984monte}.
The general notion of failure biasing means that greater probability is assigned to `\emph{failures}', i.e., transitions that are chosen with a probability that is $O(\epsilon)$.
\citet{shahabuddin1994importance} studied the asymptotic properties of a refined version of failure biasing called \emph{balanced failure biasing} (BFB), and showed the method to satisfy BRE in the absence of \rdd{so-called} High Probability Cycles (HPCs).
\citet{nakayama1996general} derived general conditions for BRE in importance sampling schemes for HRMSs.
\citet{carrasco1992failure} proposed a method called \emph{failure distance biasing}, in which the simulation measure is based on the distance from each state to the rare states. This distance notion is similar to the function $d$ discussed in Section~\ref{sec: path-based is} --- given $d$, the method applies a form of failure biasing (with the exception that if a failure does not lead to a decrease in $d$, it is not treated as a failure). The function $d$ in their setting is computed by finding the minimal cuts in the model's corresponding fault tree, which means the setting is limited (namely multicomponent systems with independent component types, and no HPCs).
\citet{carrasco2006failure} extended this approach to `unbalanced' systems.
\citet{alexopoulos2001estimating} proposed a method that is based on bounding the value of the likelihood ratios, and which has good performance for both highly reliable and highly redundant systems.
\citet{juneja2001fast} proposed a scheme --- the \emph{implementable general biasing scheme} (IGBS) --- to mitigate the effects of HPCs on the performance of BFB. 

We will use BFB and IGBS as literature benchmarks for the experiments of Section~\ref{sec: results}, so we discuss them in more detail in the following. In particular, for each state~$\state \in \ssX$, let $n_f(\state)$ be the \rda{number} of transitions leaving $\state$ with a positive $\epsilon$-order (the `\emph{failures}') and let $n_r(\state)$ be the \rda{number} of transitions leaving $\state$ with $\epsilon$-order $0$ (the `\emph{repairs}'). Given some $p>0$, the simulation measure $\Q$ of BFB is given by
\[
	q_{\statea \stateb} = \left\{ 
\begin{array}{cl} 
(n_f(\statea))^{-1} & \text{ if } n_r(\statea) = 0, \\
(n_r(\statea))^{-1} & \text{ if } n_f(\statea) = 0, \\
p (n_f(\statea))^{-1} & \text{ if } \asyrate_{\statea \stateb} > 0 \text{ and } n_r(\statea) > 0, \\
(1-p) (n_r(\statea))^{-1} & \text{ if } \asyrate_{\statea \stateb} = 0 \text{ and } n_f(\statea) > 0. 
\end{array}\right.
\]
The typical choice for $p$ is $\frac{1}{2}$, and we make the same choice in this paper. IGBS is similar to BFB, with the exception that the degree of biasing is reduced when the current state is part of an HPC. To avoid having to run a numerical procedure to detect HPCs, IGBS switches to low-intensity biasing when the previous transition was a high-probability transition (resulting in a non-Markovian simulation measure). In particular, with $q_{\statea\stateb|\statea'} = \Q(\ps{i+1} = \stateb \; | \; \ps{i} = \statea, \;\ps{i-1} = \statea')$, IGBS means:
\[
q_{\statea\stateb|\statea'}  = \left\{ 
\begin{array}{cl} 
(n_f(\statea))^{-1} & \text{ if } n_r(\statea) = 0, \\
(n_r(\statea))^{-1} & \text{ if } n_f(\statea) = 0, \\
p (n_f(\statea))^{-1} & \text{ if } \asyrate_{\statea \stateb} > 0, \asyrate_{\statea' \statea} > 0 \text{ and } n_r(\statea) > 0, \\
(1-p) (n_r(\statea))^{-1} & \text{ if } \asyrate_{\statea \stateb} = 0, \asyrate_{\statea' \statea} > 0 \text{ and } n_f(\statea) > 0, \\
\delta (n_f(\statea))^{-1} & \text{ if } \asyrate_{\statea \stateb} > 0, \asyrate_{\statea' \statea} = 0 \text{ and } n_r(\statea) > 0, \\
(1-\delta) (n_r(\statea))^{-1} & \text{ if } \asyrate_{\statea \stateb} = 0, \asyrate_{\statea' \statea} = 0 \text{ and } n_f(\statea) > 0, 
\end{array}\right.
\]
for some $\delta < p$. In the initial state, $p$ is used as a biasing constant. We choose $\delta = \frac{1}{100}$ in this paper.
Note that \rdf{the measure described above} is more general than \citet{shahabuddin1994importance}, who assumed that $\forall x \in \ssX \setminus \{\init, \gstate, \tstate\}$, $n_f(\state) > 0$ and $n_r(\state) > 0$.

In addition to the papers on Zero Variance Approximation mentioned in Section~\ref{sec: simulation}, \cite{lecuyer2011approximating} discusses the particular application of ZVA to HRMSs. We use several of the ideas therein in Section~\ref{sec: bre}. In particular, conditions are derived for a change of measure to satisfy VRE. In said paper, the analogues of $d$ and $\w$ were not obtained explicitly, but approximated using the structure of multicomponent systems. 

The \rda{basic idea underlying} Section~\ref{sec: vrff} is from \citet{juneja2007estimating}, who showed that for geometric sums of heavy-tailed random variables, a separation of the estimator into the numerical computation of a dominant component and \rda{the} simulation of the small component yields an estimator with VRE.

Other contributions involving generally applicable efficient simulation of HRMSs include \cite{budderare}, in which the notion of distance to the goal set used is the smallest possible number of transitions needed (which is equivalent to the model setting of this paper if all transitions have $\epsilon$-order 1). Another generic importance sampling technique is the \emph{cross-entropy method} (see, e.g., \citet{ridder2010asymptotic}), which we do not discuss further in this paper because of its heuristic nature.

\section{The Path-ZVA algorithm} \label{sec: dijkstra algorithm}

In this section, we describe the simulation method of this paper: Path-ZVA.
We discuss two versions: ZVA-$\ddd$ and ZVA-$\dom$, which differ in the distance measure used.
In the following, we first give a formal description of these two methods and the underlying concepts.
We then discuss their implementation, with a particular focus on the routines of Algorithms~\ref{alg: forward phase},~\ref{alg: loop detect}~and~\ref{alg: backward phase}.

\subsection{Path-based Zero Variance Approximation} \label{sec: path-based is}
\label{sec:pathzva}
Our method for finding a suitable approximation $\w$ of $\ipf$ is to select only a subset of the paths used in the summation of \eqref{eq def mu}, namely the so-called \emph{dominant} paths, as we discuss below. In order to determine which paths to select, we will define two related measures --- $\ddd$ and $\www$ --- for the distance between each state and the rare state $\gstate$. 
\rdd{Throughout this subsection, we assume that no so-called {\em High-Probability Cycle} (HPC) is present, where we define a HPC (see also Section~\ref{sec: related work}) as a cyclic path $\pth$ with~\mbox{$\pth({n_\pth})=\pth(0)$} and~$\P(\pth) = \Theta(\epsilon^0)$.
For Markov chains that do have one or more HPCs, we explain in Section~\ref{sec: explicit dijkstra} how these are removed.}

First, we define the function {$\dd:\ssX^2 \rightarrow \N$} as
\begin{equation*}
\dd(\statea,\stateb) = \min \{r : \exists \pth \in \pths(\state) \text{ with } \ps{\pthl} = \stateb, \,\,\, \forall k < \pthl : \ps{k} \notin \{\tstate, \gstate\} \; \text{ and } \; {\P}(\pth) = \Theta(\epsilon^r)\}.
\end{equation*}
Intuitively, $d(\statea,\stateb)$ is the shortest $\epsilon$-distance of any path from $\statea$ to $\stateb$.
Of particular interest are $d(\statea, \gstate)$, the shortest distance from each state $\statea \in \ssX$ to the goal state, and $d(\init,\statea)$, the shortest distance from the initial state to each state $\statea$.  

As mentioned in the introduction, we do not need to run the algorithm on the entire state space, \rdd{but} only the states that are asymptotically at most as hard to reach from $\init$ as $\gstate$, and \rdd{their} neighbours. To formalise this, we introduce the following two sets:
\begin{equation} \label{eq def core}
\core = \{\state \in \ssX : d(\init,\state) \leq d(\init, \gstate)\} \,\,\,\, \text{  and  }
\end{equation}
\vspace{-0.7cm}
\begin{equation*}
\edge = \{\state \in \ssX \setminus \core: \exists \stateb \in \core \text{ s.t. } p_{\stateb \state} > 0\}.
\end{equation*}
In words: $\core$ is the relevant part of $\ssX$, i.e., the set of states that are asymptotically not substantially less likely to \rdd{be reached} from $\init$ than $\gstate$. The set $\edge$ contains the states `bordering' $\core$, i.e., those states to which the system can jump directly from $\core$. \rdd{By construction, $d(s, x) > d(s, g)$ for all $x \in \edge$. } We assume that both \rdd{$\core$ and $\edge$} are finite --- if they are not, the numerical pre-processing \rdd{phase} will never terminate.

{The algorithm of this {paper} calculates $\dd(\init, \state)$, $\dd(\state, \gstate)$, and} $\w(\state)$ only for $\state \in \core$. 
This means that~\rdb{\eqref{eq zva def}} \rdb{cannot be applied when $\state \in \core$ and $\stateb \in \edge$}. This is remedied by \rdb{adapting $\P$ to an} alternative probability measure $\PPP$ \rdb{with high-probability `shortcuts' from $\edge$ to $\gstate$}, and \rdb{its corresponding} distance measure $\ddd$.
\rdb{First let $\ppp_{\statea\stateb}$ be defined as follows:}
\begin{equation}
  \ppp_{\statea\stateb} = \left\{\begin{array}{ccr} p_{\statea\stateb} & & \text{if } \statea \notin \edge,\\ 1 & & \text{if } \statea \in \Gamma \text{ and } \stateb = \gstate, \\ 0 & & \text{otherwise.}\end{array}\right. \label{eq def prf}
\end{equation}
\rdb{Then we let $\PPP$ and $\ddd$ be defined as $\P$ and $d$ under this new measure. Next, we define}
\begin{equation*}
\dom(\state) = \{\pth \in \ipths(\state): {\PPP}(\pth) = \Theta(\epsilon^{\ddd(\state,\gstate)})\},
\end{equation*}
 the set of paths from $\state$ to the goal state $\gstate$ that have \rdd{(under $\PPP$)} the minimal distance $\ddd(\state, \gstate)$. \rdd{As before, we compute $\ddd(\init, \state)$, $\ddd(\state, \gstate)$, and $\dom(\state)$ only for state $\state$ if $\state \in \core$.} (Note that, even though we allow $\pths(\state)$ and $\ipths(\state)$ to include `paths' that have probability zero under $\P$, such paths are never in $\Delta(x)$, since either they include one or more transitions with probability zero under $\PPP$, or they traverse $\Gamma$ and their $\epsilon$-order exceeds $\ddd(\state, \gstate)$.)
We call the paths in $\Delta(x)$ the \emph{dominant} paths from $\state$ to the goal state.
\begin{figure}
	\caption{Illustration of the sets $\core$ and $\edge$ within $\ssX$.}
	\begin{center}
		\scalebox{0.4}{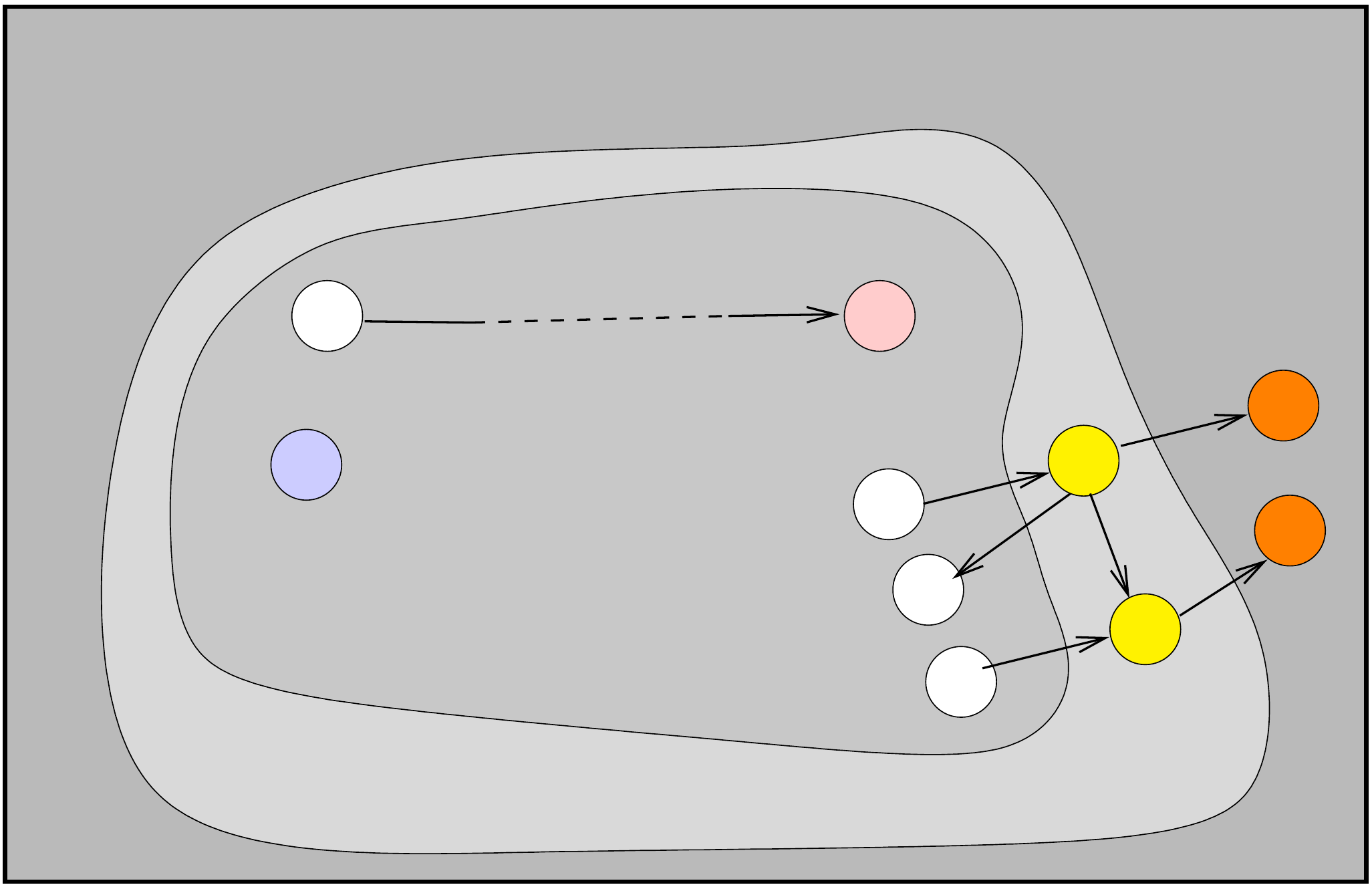}
	\end{center}
\end{figure}
Finally, we define the function $\www:\ssX \rightarrow \mathbb{R}^+$ as the probability of the `dominant' paths under $\PPP$:
\begin{equation}
\www(\state) = \sum_{\pth \, \in \, \dom(\state)} {\PPP}(\pth).
\label{eqdefvdom}
\end{equation}
The function $\www$ can be substituted for $\w$ in \eqref{eq zva def} to yield a well-performing simulation
measure.  This approach will be called ZVA-$\dom$ in this paper.
Alternatively, one can \rdd{use} $\w(x)=\epsilon^{\ddd(\state, \rdd{\gstate})}$, which is easier
to compute and, as we will see in Section~\ref{sec: bre}, still yields an estimator with favourable properties.
This approach will be called ZVA-$\ddd$ in this paper.
Note that there are model settings for which techniques exist that allow for ZVA-$\ddd$ to be applied without the need to consider each individual state in~$\core$: see, e.g., \citet{reijsbergen2013automated} for an application to stochastic Petri nets. In the approach of that paper, the full state space is partitioned into `zones' such that for each zone it holds that $\ddd$ in each state is given by the same affine function of the state vector.
The performances of ZVA-$\dom$ and ZVA-$\ddd$ will be compared in Section \ref{sec: results}.

Regardless of the choice of $\w$, when we leave $\core$ during the simulation we \emph{stop using importance sampling} and revert back to standard Monte Carlo \rdb{until we reach either $\gstate$ or $\tstate$}. 
A consequence is that the simulation measure $\Q$ is now \emph{non-Markovian}: it is {only} Markovian \rdb{as long as we stay in} $\core$.
Let
\begin{equation} 
\stopis(\pth) = \min \{i \in \N : \ps{i} \notin \Lambda \mbox{ or } \ps{i} = \gstate\}.
\label{eq stopis}
\end{equation}
 Then $\Q$ is as follows \rdb{(replacing \eqref{eq zva def})}:
\begin{equation}
q_{\ps{i} \ps{i+1}} = \left\{ 
\begin{array}{cl} 
\frac{\displaystyle \rdb{\ppp}_{\ps{i} \ps{i+1}} \w(\ps{i+1})}{\displaystyle \sum_{\stateb \in \ssX} \rdb{\ppp}_{\ps{i} \stateb} \w(\stateb)} & \text{ if } i < \stopis(\pth)
\\ 
p_{\ps{i} \ps{i+1}} & \text{ if } i \geq \stopis(\pth) .
\end{array}\right. \label{eq ch4 com}
\end{equation}

\subsection{\rdc{Pre-processing: graph analysis procedure for finding $\ddd$ and $\www$}} \label{sec: explicit dijkstra}
\label{sec:graphalgorithms}

The algorithm for determining $\ddd$ and $\www$ involves the search for a shortest path in a graph, and is strongly inspired by Dijkstra's method \cite{dijkstra1959note}. The \rda{new} algorithm can be broken down into three main routines, namely Algorithms~\ref{alg: forward phase}, \ref{alg: loop detect}~and~\ref{alg: backward phase}. Unlike Dijkstra's algorithm, the algorithm of this section consists of two phases: a forward phase and a backward phase. \rda{In the forward phase, we generate the state space and remove HPCs until we have found $\gstate$ and $\core$, and in the backward phase we start in $\gstate$ and determine $\ddd$ and $\www$ by working back until we reach $\init$}. The forward phase is described in Algorithm~\ref{alg: forward phase} and the backward phase is described in Algorithm~\ref{alg: backward phase}. Algorithm~\ref{alg: loop detect} removes a detected \rdd{HPC} and is called by Algorithm~\ref{alg: forward phase}. \rdc{The run times of all the algorithms are polynomial in the size of $\core \cup \edge$}.

\subsubsection{Forward phase} \label{sec: forward phase}

In the first phase, we use a procedure based on Dijkstra's algorithm for finding shortest paths in a graph in order to determine {$\ddd(\init,\cdot)$,} $\core$ and to remove all HPCs. In particular, $\ddd(\init,\cdot)$ is used to detect the HPCs; it is denoted by \rda{$\ddd'(\cdot)$} in Algorithm~\ref{alg: forward phase} \rda{for brevity}.

\begin{algorithm} 
\caption{Forward phase. } 
\label{alg: forward phase}  
{\fontsize{10}{10}\selectfont
\begin{algorithmic}[1]
    \Require Markov chain $(\ssX,P)$ with $P = (p_{\statea \stateb})_{\statea, \stateb \in \ssX}$, source $\init$, destination $\gstate$.
    \State ${\core} := \rda{\emptyset}$
		\State $\rdc{\ddd'}(\init) := 0$, \;\; $\forall {\stateb} \in \ssX \setminus \{\init\}: \rdc{\ddd'}({\stateb}) := \infty$ \label{ln: mvc init d} 
    \State $P':= P$, \;\; ${\state} := \init$
    \While{{$\rdc{\ddd'}( x) \leq \rdc{\ddd'}({\gstate})$} } \label{ln: mvc main while}
				\State $\core := \core \cup \{{\state}\}$
    		\ForAll{${\stateb} \in \ssX$ {s.t. $p_{\statea \stateb}>0$}} \label{ln: mvc main forall}
    	      \State $\rdc{\ddd'}({\stateb}) := \min(\rdc{\ddd'}({\stateb}), \rdc{\ddd'}({\state})+r_{{\state}{\stateb}})$
    	      \If{${\stateb} \in \core \text{ and } \rdc{\ddd'}({\stateb}) = \rdc{\ddd'}({\state})$} \label{ln: loop detect if}
    	      		\State $P' := $ loopDetect($(\ssX,P'),{\stateb}$) \label{ln: call loop detect}
    	      \EndIf
        \EndFor
        \State ${\state} := \arg\min \{\rdc{\ddd'}({\stateb}): {\stateb} \in \ssX \setminus \core\} $ \label{ln: choose one}  \Comment{if several states are possible} 
    \EndWhile \Comment{in line~\ref{ln: choose one}, any can be chosen}
    \State \Return $\core, P'$
\end{algorithmic}
}
\end{algorithm}

\begin{algorithm} 
\caption{loopDetect($(\ssX,P),\state'$).} 
\label{alg: loop detect}  
{\fontsize{10}{10}\selectfont
\begin{algorithmic}[1]
    \Require Markov chain $(\ssX,P)$, state $\state'$.
    \State $P':= P$
    \State $A := \emptyset, B := \emptyset$
		\State \rda{$S_A := \{\state'\}, S_B := \{\state'\}$}
		\While{ ${S_A} \neq \emptyset \text{ and } {S_B} \neq \emptyset$ } \label{ln: loopAB}
			\State {$A := A \cup S_A, B := B \cup S_B$}
		  \State {$S'_A := S_A$}, $S_A := \{ \stateb \in \ssX \setminus A  \; : \; \exists \statea \in S'_A$ s.t. $\asyrate_{\statea \stateb} = 0\}$
		  \State {$S'_B := S_B$}, $S_B := \{ \stateb \in \ssX \setminus B  \; : \; \exists \statea \in S'_B$ s.t. $\asyrate_{\stateb \statea} = 0\}$
		\EndWhile \label{ln: loopAB end}
		\While{ ${S_A} \neq \emptyset$ } \label{ln: loopA}
			\State {$A := A \cup S_A$}
		  \State {$S'_A := S_A$}, $S_A := \{ \stateb \in B \setminus A  \; : \; \exists \statea \in S'_A$ s.t. $\asyrate_{\statea \stateb} = 0\}$
		\EndWhile \label{ln: loopA end}
		\While{ ${S_B} \neq \emptyset$ } \label{ln: loopB}
		 \State {$B := B \cup S_B$}
		  \State {$S'_B := S_B$}, $S_B := \{ \stateb \in A \setminus B  \; : \; \exists \statea \in S'_B$ s.t. $\asyrate_{\stateb \statea} = 0\}$
		\EndWhile \label{ln: loopB end}
    \State $L := A \cap B$ \label{ln: def L}
    \State $D := \{\statea \in \ssX {\setminus L}: \exists \stateb \in L$ s.t. $p'_{\stateb \statea}>0\}$
    \State Solve $\left[\begin{array}{rll} \mu_{\statea \stateb} & = p_{\statea \stateb}+\sum_{\stateb' \in L}  p_{\statea \stateb'} \mu_{\stateb' \stateb}, & \forall \statea \in L, \stateb \in D, \\ 1 & = \sum_{\stateb' \in D} \mu_{\statea\stateb'}, & \forall \statea \in L \end{array}\right]$ for $\mu_{\statea \stateb}$ \label{ln: scc merge}
    \State $p'_{\statea \stateb} := \mu_{\statea \stateb}, \forall \statea \in L, \stateb \in D$, \;\; {$p'_{\statea \stateb} := 0, \forall \statea \in L, \stateb \in L$}
    \State \Return $P'$.
\end{algorithmic}
}
\end{algorithm}

 Whilst running the procedure, we iteratively update $\core$ --- this allows us to use $\core$ to keep track of the visited states. We initialise \rdc{$\core = \emptyset$} and $\rdc{\ddd'}(\init)=0$. We set the current state ${\state}$ equal to $\init$. Then, we carry out the following routine until ${\state}$ equals $\gstate$: we \rdc{add ${\state}$} to $\core$, and set $\rdc{\ddd'}({\stateb}) = \min(\rdc{\ddd'}({\stateb}), \rdc{\ddd'}({\state}) + r_{{\state} {\stateb}})$ \rdc{for each possible successor state ${\stateb}$ of ${\state}$} {--- i.e., we let the new best value for $\rdc{\ddd'}(\rdc{\stateb})$ be the minimum between the old best value and the new possible value}. We then set ${\state}$ equal to the \rdc{state~$\stateb$} that has not been considered before with the lowest value of $\rdc{\ddd'}$, and start over. When we have reached $\gstate$, we complete the procedure for all states ${\stateb}$ with $\rdc{\ddd'}({\stateb}) = \rdc{\ddd'}(\gstate)$ before we terminate the first phase. The set $\core$ then meets its definition given in \eqref{eq def core}.

If, whilst running the procedure, we find that a state ${\stateb}$ has a successor state ${\stateb}'$ such that $\rdc{\ddd'}({\stateb}) = \rdc{\ddd'}({\stateb}')$, we trigger \rdc{the} loop-detection procedure \rdc{of Algorithm~\ref{alg: loop detect}}. 
\rdd{It} essentially boils down to removing all low-probability transitions from the relevant part of the DTMC and finding the Strongly Connected Component (SCC) that contains the states ${\stateb}$ and ${{\stateb}'}$ that triggered the procedure, using the algorithm from \citet{barnat2011distributed}. Essentially, we determine $A$, {the set of states that can be reached} from ${\stateb}$ using high-probability transitions{,} and $B$, {the set of} states {from which} ${\stateb}$ {can be reached} using high-probability transitions. The relevant SCC is then $A \cap B$. 

\rdc{In Algorithm~\ref{alg: loop detect}, we find $A$ through the set $S_A$ which contains those states added to $A$ in each step. After initialising $S_A = \{\state\}$, we iteratively find those states that can be reached from the states in the previous iteration of $S_A$ (denoted by $S'_A$ in the algorithm) using high-probability transitions. We terminate when no more states can be added, i.e., when $S_A$ equals $\emptyset$. This is done in lines~\ref{ln: loopA}-\ref{ln: loopA end}; we do something similar for $B$, $S_B$ and $S'_B$ in lines~\ref{ln: loopB}-\ref{ln: loopB end}. These lines are preceded by lines~\ref{ln: loopAB}-\ref{ln: loopAB end} in which we combine $A$ and $B$. The reason behind this combined phase is that} $B$ is potentially (much) larger than $\core$ and $\edge$; it may even be infinite. In order to avoid the algorithm{'s non-termination} due to this complication we alternate between carrying out a step for $A$ and a step for $B$ \rdc{in lines~\ref{ln: loopAB}-\ref{ln: loopAB end}}. {If we can no longer find new candidates for $A$, then $A$ has been determined. Since states in the HPC need to be both in $A$ and $B$, we from then on only select candidates for $B$ that are in $A$.  We terminate if we can no longer find candidates for $B$ in $A$. The same is done for $A$ and $B$ interchanged.} This way, we always terminate in a finite amount of time {because $A \subset \core$ and $\core$ is finite}.

\begin{algorithm} 
\caption{Backward phase.} 
\label{alg: backward phase}  
{\fontsize{10}{10}\selectfont
\begin{algorithmic}[1]
    \Require Markov chain $(\core,P')$, end node $\gstate$.
    \State $\forall\stateb \in \core: $ $\,\, {\rda{\www}}(\stateb) := 0, {\rdc{\ddd}^*(\stateb) := \infty}$
    \State $\rda{\www}(\gstate) := 1, {\rdc{\ddd}^*(\gstate) := 0}$
    \State $\core' := \emptyset$, ${\statea} := \gstate$
		\State {$\edge := \{\state \in \ssX \setminus \core: \exists \stateb \in \core \text{ s.t. } p_{\stateb \state} > 0\}$}
    \While{{$\core' \neq {\core \cup \edge}$}} \label{ln: bkw main while}
		    		\State $\statea := \arg\min\{\rdc{\ddd}^*({\statea}) : {\statea} \in {(\core \cup \edge)} \setminus \core' \text{ and } \nexists {\statea}' \in {(\core \cup \edge)} \setminus \core'$ s.t. $\asyrate_{{{\statea' \statea}}} = 0$\} \label{ln: choose min} 
    		\ForAll{$\stateb \in {\core \cup \edge}$}  \label{ln: bkw main forall}\Comment{if several states are possible} 
				  \State {${\bf if} \,\,\,  \asyrate_{\stateb \statea} + \rdc{\ddd}^*(\statea) < \rdc{\ddd}^*(\stateb) \,\,\, {\bf then} \,\,\, \rda{\www}({\stateb}) := 0$} \Comment{in line~\ref{ln: choose min}, any can be chosen.}
					\State {$\rdc{\ddd}^*(\stateb) := \min(\rdc{\ddd}^*(\stateb),\asyrate_{\stateb \statea} + \rdc{\ddd}^*(\statea))$}
    				\If{$\rdc{\ddd}^*(\stateb) = \rdc{\ddd}^*({\statea}) + \asyrate_{\stateb {\statea}}$} \label{ln: if short}
    				    \State $\rda{\www}(\stateb) := \rda{\www}(\stateb) + p'_{\stateb {\statea}} \rda{\www}({\statea})$ \label{ln: update w}
    				\EndIf
    		\EndFor
    	  \State 	$\core' := \core' \cup {\{ \statea \}}$
    \EndWhile
    \State \Return ${\ddd}^*, \rda{\www}, \edge$.
\end{algorithmic}
}
\end{algorithm}

Having determined the SCC, we construct a new DTMC with the same state space and identical rare event probabilities $\ipf({\state})$ $\forall \state \in \ssX$, but with the transition probabilities of the states in the HPC redistributed. This can be done using a SCC-based state space reduction technique similar to the one described by \citet{abraham2010dtmc}, implemented in \rdc{l}ine~\ref{ln: scc merge} of Algorithm~\ref{alg: loop detect}. In our implementation, the system of equations in  \rdc{l}ine~\ref{ln: scc merge} is approximately solved using Gauss-Seidel. \rdc{Algorithm~\ref{alg: loop detect}} is repeated each time a new HPC is detected.

\subsubsection{Backward phase} \label{sec: backward phase}

In this phase, we determine $\www$ and $\ddd(\cdot,\gstate)$; the latter is denoted by \rda{$\ddd^*(\cdot)$} in Algorithm~\ref{alg: backward phase}.
We initiate the second phase in $\gstate$  (since $\gstate$ is given implicitly through a high-level description, this would not have been possible without the first phase). 
We use a list $\core'$ to keep track of the states that have been considered, and initialise $\core', \www$, and $\ddd$ as outlined in the beginning of Algorithm~\ref{alg: backward phase}. For each predecessor $\state$ of $\gstate$ that is in ${\core \cup \edge}$, we add $\state$ to $\core'$ if this had not been {done} already {and if \mbox{$\rdc{\ddd}(\state) = \asyrate_{\state \gstate}$}} we update $\rda{\www}(\state) := \rda{\www}(\state) + p'_{\state \gstate}$. We then choose the next state to consider: this is the state $\state$ in ${(\core \cup \edge)} \setminus \core'$ (i.e., {the set of states that have not yet been considered}) {for which $\rdc{\ddd}$ is the lowest} \emph{and} for which no \rda{other} state $\stateb$ {in} ${{(\core \cup \edge)}\setminus \core'}$ exists for which ${\asyrate}_{\state \stateb} = 0$. The reason is that otherwise, the probability of the paths going from $\state$ to $\stateb$ is never added to $\www(\state)$, which has a cascading effect on the predecessors of $\state$. Note that we can always find such a state only if the HPCs have been removed.
We continue performing the same procedure until we have determined $\www(\state)$ for all $\state \in {\core \cup \edge}$.

\subsection{Practical Aspects of the Path-ZVA Algorithm} \label{sec: practical application}

\subsubsection*{Identifying $\epsilon$ in practical models}

In principle, the algorithms described above can be applied to any DTMC with transition probabilities
that are parameterised by powers of some small parameter~$\epsilon$.
Usage of $\epsilon$-powers for the purpose of analysing the efficiency of simulation algorithms
goes back to at least~\citet{shahabuddin1994importance}.
However, in our case (and earlier, see~\citet{deboer2007estimating}) the change of measure {\em itself} depends on the $\epsilon$-powers.
This means that a practitioner who has a model with given rates/probabilities will
need to assign $\epsilon$-powers to them, which can be done in infinitely
many ways.

There are a few trivial approaches that do not work well, but are illustrative.
One is to simply set the $\epsilon$-power to 0 for all transitions,
and represent the model entirely by the pre-factors $\lambda_{\state\stateb} = p_{\state\stateb} / \epsilon^{\asyrate_{\state\stateb}}$. Then our algorithm will
treat the model as one large HPC, and the probability of interest will be computed numerically
if the state space is sufficiently small.
The other extreme is to set all pre-factors to~1, choose a value of $\epsilon$
just below~1, and represent the model entirely by (very high) exponents~$r_{xz}$.
Then the algorithm will focus the simulation effort on the single most likely path,
at the expense of paths which are only slightly (namely by a factor of $\epsilon$)
less likely, causing underestimation and/or high variance.
A third approach is to set all $\epsilon$-powers to 1, as is done by \citet{budderare}. Although
this is a more natural approach than the other two, it still does not distinguish between failures and repairs.

In typical reliability models, repair rates are several orders of magnitude higher than failure rates.
In such cases, giving component repairs $\epsilon$-order 0 and failures $\epsilon$-order 1
is typically a good choice.
If some failures are very much less likely than others (this is a feature of so-called `unbalanced' systems), higher $\epsilon$-orders
can be assigned to those to achieve further variance reduction
(see \citet[Fig.~1]{shahabuddin1994importance}).
This approach can be automated to a large extent by having the practitioner specify only $\epsilon$ beforehand,
and assigning the smallest integer $\epsilon$-power to each transition such that its pre-factor is greater than $\epsilon$. This is in fact what we have implemented and applied in Section~\ref{sec: realistic examples}.
\citet{carrasco2006failure} chooses $\epsilon$ as the ratio of the largest failure rate to the smallest repair rate. Further experimentation to establish best practice with regards to choosing $\epsilon$ is an interesting direction for further research.

\subsubsection*{Numerical Complexity}

The numerical complexity of the phases of our algorithm is as follows. Let $D$ be the maximum number of successors of all states in $\Lambda$ (this is $|\Lambda\cup\Gamma|$ at worst but usually much smaller). The loop in line~\ref{ln: mvc main while} of Algorithm~\ref{alg: forward phase} has $|\core|$ iterations, and the nested loop in line~\ref{ln: mvc main forall} has $\dim$ iterations, so the total complexity is $O(D|\core|)$. Lines \ref{ln: loopAB}-\ref{ln: loopB end} of Algorithm~\ref{alg: loop detect} have complexity $O(\max |L|)$, where $\max |L|$ denotes the size of the largest HPC plus direct predecessors and successors. Line~\ref{ln: scc merge} of Algorithm~\ref{alg: loop detect} has a complexity of $O((\max |L|)^2)$ if implemented using the approximative Gauss-Seidel algorithm. Line~\ref{ln: bkw main while} of Algorithm~\ref{alg: backward phase} has $|\core \cup \edge|$ iterations, and although the nested loop in line~\ref{ln: bkw main forall} only has to be done for the number of predecessors in each state, these two loops together will have total complexity $O(D|\core \cup \edge|)$ since the total number of incoming and outgoing transitions within $\core\cup\edge$ is the same.

In summary, the complexity of our algorithm is typically $O(D|\core \cup \edge|)$ or $O(|\core \cup \edge|^2)$.
\rdf{
This is to be compared to the cost of computing the probability of interest without simulation,
which is typically $O(D|\mathcal{U}|)$ or $O(|\mathcal{U}|^2)$,
where $\mathcal{U}$ is what remains of the full state space~$\ssX$ after collapsing
all goal states (and states that can only be reached via goal states) into a single state~$\gstate$.
}
Hence, what we gain is that we apply numerical analysis only to \mbox{$\core \cup \edge$}
\rdf{
rather than to~$\mathcal{U}$.
}
This is illustrated in Table~\ref{tab: numerical summary} for a range of models.

\begin{table}[!t]
 \centering
\tbl{Total and reduced state space sizes and the pre-processing sets $\core$ for a range of models.
More information can be found in the following sources: (R) Reijsbergen et al.\ [2013], (S) Section 6.2.1, (A) Alexopoulos and Shultes [2001], (C) Carrasco [2006].
\rdf{The $>$ 500\,000 entries for $|\mathcal{U}|$ are lower bounds established by 12 hours of computation.}%
}{
\begin{tabular}{|lc|cccc|}
\hline
\rule{0pt}{1.1em}%
Model & Source & $|\ssX|$ (total) & $|\mathcal{U}|$ & $|\core|$ & $|\edge|$ \\ \hline
\rule{0pt}{1.1em}%
$2$-node tandem queue, overflow level $n$    & (R)       & $\infty$ & $\infty$ & $O(n^2)$ & $O(n)$ \\
Distrib.\ Datab.\ Syst.\ (dedicated repair)  & many; (S) & 421\,875 & 514 & 48 & 84 \\
Distributed Database System (FCFS)           & see (S)   & 2\,123\,047\,371 & $>$ 500\,000 & 84 & 504 \\
$k$-out-of-$n$ system (homogeneous)          & (A)       & $O(n)$ & $O(k)$  & $O(k)$  & 0 \\
$k$-out-of-$n$ system (heterogeneous)        &           & $O(2^n)$ & $O(2^k)$ & $O(2^k)$ & 0 \\
Fault-Tolerant Database System               & (C)       & 14\,762\,250\,000 & 59\,051 & 87 & 1060 \\
Fault-Tolerant Control System                & (C)       &  1\,855\,425\,871\,872 & $>$ 500\,000& 116 & 2928 \\
Network with Redundancies                    & (A)       & very large & very large & \multicolumn{2}{c|}{still very large} \\
\hline
\end{tabular}}
\label{tab: numerical summary}
\end{table}

\subsubsection*{\rdf{High component reliability vs.\ high redundancy}}
\rdf{
For models whose high reliability is mostly due to high redundancy, the method tends to be less effective.
One reason is that $|\Lambda|$ is large in such models; this is apparent
in the last line, and potentially also the fifth line (depending on the value of $k$), of
Table~\ref{tab: numerical summary}.
The other reason is that when many `almost-dominant' paths exist, of
order $\epsilon^{\bar d(s,g)+1}$ or higher, their total contribution may dominate the (fewer) supposedly `dominant'
path(s) of order $\epsilon^{d(s,g)}$, if $\epsilon$ is not small enough.
This can easily happen in models of highly-redundant systems,
with e.g.\ many different possible sequences of failure and repair events
  on those almost-dominant paths, and
  $\epsilon$ tending to be larger because of larger 
  individual component failure rates. 
}

\subsubsection*{Efficient implementation}

A crude way of implementing the method would involve constructing the
entire state space and keeping track of matrices giving the transition probabilities
and $\epsilon$ powers for each combination of states.
However, this would be very memory-inefficient, or impossible in case of an infinite state space.
Specification of a model in our implementation consists only of three functions that determine, given a state:
(1) whether it is a goal state, (2) whether it is a taboo state,
and (3) three arrays specifying its successors' state indices,
the probabilities of jumping to these successors (typically implicitly through CTMC rates),
and the corresponding $\epsilon$-powers. Our implementation also allows for the last array to be omitted and the $\epsilon$-orders to be computed using a given value $\epsilon$ in the manner discussed previously.
There is no need to generate the entire state space; states only need to be considered
`on the fly', as they are encountered during pre-processing and the actual simulation.

\section{Asymptotic Performance of the Estimator} \label{sec: bre}

In this section, we consider the performance of the two versions of the estimator produced by the algorithm of Section~\ref{sec: dijkstra algorithm}.
If the estimator is based on $\www$, we show it has VRE 
(Theorem~\ref{th:vre}); if it is based on $\ddd$ (which is easier to compute), it does 
not necessarily have VRE, but it does have both BRE (Theorem~\ref{th:bre}) and 
the `Bounded Normal Approximation' property (Theorem~\ref{th:bna})
We first prove the technical {L}emmas~\ref{lm: l2}-\ref{thm: bounded moments} before proving the main theorems.

\begin{lemma} \label{lm: l2}
If $\w(\state) = \Theta(\epsilon^{\ddd(\state, \gstate)})$ uniformly in $x$,
then for all $\state \in \core$ we have that
\begin{equation*}
\sum_{z \in \ssX} \ppp_{\state \stateb} \w(\stateb) = \sum_{\stateb \in \core \cup \edge} \ppp_{\state \stateb} \Theta(\epsilon^{\ddd(\stateb, \gstate)}) = \Theta(\epsilon^{\ddd(\state, \gstate)})
\end{equation*}
uniformly in $x$.
\end{lemma}
\begin{proof}
Since \rdb{$\ppp_{\statea \stateb} = 0$ for $\stateb \notin \core \cup \edge$ and since} $\core \cup \edge$ \rdb{is} finite,
the $\epsilon$-order of the sum equals the $\epsilon$-order of its largest element.
Let $\stateb'$ be \rdb{a} state such that $\ppp_{\state \stateb'} \Theta(\epsilon^{\ddd(\stateb', \gstate)})$
has the lowest $\epsilon$-order in the sum.
\rdb{Suppose} that its $\epsilon$-order is smaller than $\ddd(\state, \gstate)$,
then there exists a path from $\state$ via $\stateb'$ to $\gstate$ with cost lower than $\ddd(\state, \gstate)$,
which contradicts the definition of $\ddd(\state, \gstate)$\rdb{.}
The uniformity follows trivially from the finite\rdb{ness of $\core \cup \edge$}.
\end{proof}

\begin{lemma} \label{lm: l2x}
\begin{equation*}
d(\init, \gstate) = \ddd(\init, \gstate)
\end{equation*}
\end{lemma}
\begin{proof}
By \rdb{the} definition \rdb{of $\core$}, any state $x\notin \core$ has \rdb{$d(s,x) > d(s,g)$} and \mbox{\rdb{$\ddd(s,x) > d(s,g)$}},
so any path leaving $\core$ has length \rdb{$> d(s,g)$}, both under $\P$ and $\PPP$.
Therefore, the shortest path from $s$ to $g$ under $\P$ must lie entirely inside $\core$,
and its length under $\PPP$ is $d(s,g)$ too.
\rdb{Finally}, any other path from $s$ to $g$ under $\PPP$ cannot be shorter than $d(s,g)$:
if it doesn't leave $\core$, its length is the same under $\PPP$ and $\P$,
while if it leaves $\core$, its length exceeds $d(s,g)$.
\end{proof}

\begin{lemma}  \label{thm: new lemma}
If $\w(\state) = \Theta(\epsilon^{\ddd(\state, \gstate)})$ uniformly in $\state$,
then for any path $\omega$ starting in $s$ and ending in $g$ or $\Gamma$ before leaving $\Lambda\cup \Gamma$,
we have
\[
  \frac{\P(\pth)}{\Q(\pth)} = \Theta(\epsilon^{d(\init, \gstate)})
\quad\text{and, more specifically,}\quad
  \frac{\P(\pth)}{\Q(\pth)}  \leq c_0^r \epsilon^{d(\init, \gstate)} 
\]
for some positive $c_0$, independent of $\omega$,
and with $r$ the epsilon-order of $\omega$.
\end{lemma}

\begin{proof}
Observe that
\begin{eqnarray*}
\frac{\P(\pth)}{\Q(\pth)}
= \prod_{i=1}^{\pthl} \frac{\ppp_{{\ps{i-1}}{\ps{i}}}}{q_{{\ps{i-1}}{\ps{i}}}} 
& = & \prod_{i=1}^{\pthl} \frac{{\sum_{\stateb \in \ssX} \ppp_{\ps{i-1} \stateb} v(\stateb)}}{{v(\ps{i})}} \\
& = & \prod_{i=1}^{\pthl} \frac{\Theta(\epsilon^{\ddd(\ps{i-1}, \gstate)})}{\Theta(\epsilon^{\ddd(\ps{i}, \gstate)})}
  = \frac{\Theta(\epsilon^{\ddd(\init, \gstate)})}{\Theta(\epsilon^{\ddd(\ps{\pthl}, \gstate)})}
  = \Theta(\epsilon^{d(\init, \gstate)}).
\end{eqnarray*}
The second equality follows directly from \eqref{eq ch4 com},
the third equality from the lemma's assumption and Lemma~\ref{lm: l2},
and the last equality from Lemma~\ref{lm: l2x}. 

The second more specific result follows by observing that since the set $\Lambda$ is finite
and contains no high-probability cycles, there is an upper bound on how much likelihood ratio can
be accumulated between between two transitions of $\epsilon$-order $\geq 1$.
\end{proof}

\begin{lemma}  \label{thm: bounded moments}
If $v(\state) = \Theta(\epsilon^{\ddd(\state, \gstate)})$ \rdb{uniformly in $\state$}, then with $\Q(\pth)$
according to~\eqref{eq ch4 com}, we have \rdd{for any real-valued $k \geq 1$}
\[
  \E_{\Q} (\likrat^k \cdot {\bf 1}_{\tot} ) = \Theta(\epsilon^{k d(\init, \gstate)}).
\]
\end{lemma}

\begin{proof}
Start by calculating an upper bound on the $k$'th moment (see below for explanation):
\begin{equation}
\begin{split}
\label{eqgeomseriesarg}
\E_{\Q} (\likrat^k \cdot {\bf 1}_{\tot} )  
&= \sum_{\pth \in \ipths(\init)} \Q(\pth) \left( \frac{\P(\pth)}{\Q(\pth)}\right)^k
= \sum_{r=d(s,g)}^\infty \sum_{\pth \in \ipths^{r} (s)} \P(\pth) \left( \frac{\P(\pth)}{\Q(\pth)}\right)^{k-1}
\\&
\leq \sum_{r=d(s,g)}^\infty \sum_{\pth \in \bar\Phi^{r} (s)} \P(\pth) \left( \frac{\P(\pth)}{\Q(\pth)}\right)^{k-1}
\\&
\leq \sum_{r=d(s,g)}^\infty \sum_{\pth \in \bar\Phi^{r} (s)} \P(\pth) \left( c_0^r \epsilon^{d(s,g)} \right)^{k-1}
\\&
= \epsilon^{d(s,g)\cdot(k-1)} \sum_{r=d(s,g)}^\infty \P(\bar\Phi^r) \left( c_0^r \right)^{k-1}
\leq c_1 \epsilon^{k d(s,g)}
.
\\
\end{split}
\end{equation}
where $c_0$ and $c_1$ are positive constants,
and $\bar\Phi^r$ is like $\Phi^r$, but with paths ending at their first visit to $\{g\} \cup \Gamma$ rather than at~$g$.
Since paths reaching or passing through~$\Gamma$ have at least $\epsilon$-order $d(s,g)+1$ by definition of $\Gamma$,
it follows that for any path $\omega \in \cup_{r\geq d(s,g)+1} \Phi^r(s)$,
the path $\omega' = (\omega_0, \omega_1, \dots, \omega_{m(\omega)})$ is in $\cup_{r\geq d(s,g)+1} \bar\Phi^r(s)$,
with $m(\omega)$ as defined in \eqref{eq stopis};
and together with $\P/\Q = 1$ for steps on a path beyond $\Gamma$, this motivates the first inequality.
The second inequality follows from Lemma~\ref{thm: new lemma}.
The third inequality is established by observing that $\P(\bar\Phi^r(s)) = \Theta(\epsilon^r)$,
which is not trivial, since an infinite number of subdominant paths
could conceivably contribute more than something that is $\Theta(\epsilon^r)$, but the bound follows from
the finiteness of and the absence of HPCs in $\Lambda\cup\Gamma$, and a geometric series argument as used in the
proof of Theorem 1 of \citet{lecuyer2011approximating} (and in Lemma 5.6 of \citet{reijsbergen2013efficient}).

A lower bound on the $k$'th moment is found by restricting the summation to only the dominant paths:
\begin{equation}
\E_{\Q} (\likrat^k \cdot {\bf 1}_{\tot} ) 
\geq \sum_{\pth \in \dom(\init)} \Q(\pth) \left( \frac{\P(\pth)}{\Q(\pth)}\right)^k
\geq c_2 \epsilon^{k d(\init, \gstate)} \sum_{\pth \in \dom(\init)} \Q(\pth)
\geq c_3 \epsilon^{k d(\init, \gstate)}
\label{eq l5e}
\end{equation}
where the last equality uses Lemma~\ref{thm: new lemma}, in essence saying that
under $\Q$, the dominant paths have total probability $\Theta(1)$. \rdb{In \eqref{eq l5e}, $c_2$ and $c_3$ are positive constants.}
\end{proof}

\begin{theorem}
\label{th:bre}
If $v(x) = \Theta(\epsilon^{\ddd(x, \gstate)})$ \rdb{uniformly in $\state$}, then 
the estimator based on $v$ and $\Q$ according to~\eqref{eq ch4 com}
has BRE:
\[
\frac{\vr_{\mathbb{Q}}(\likrat \cdot {\bf 1}_{\tot})}{\mathbb{E}^2_{\mathbb{Q}}(\likrat \cdot {\bf 1}_{\tot})} = O(1). 
\]
\end{theorem}

\begin{proof}
	\rdd{Immediate by using $\vr(X)=\mathbb{E}X^2 - \mathbb{E}^2 X$ and applying Lemma~\ref{thm: bounded moments}.}
\end{proof}


\begin{theorem}
\label{th:vre}
If $v(x) =\sum_{\pth \in \dom(x)} \PPP(\pth)$, then 
the estimator based on $v$ and $\Q$ according to~\eqref{eq ch4 com}
has VRE:
\[
\lim_{\epsilon\downarrow 0}
\frac{\vr_{\mathbb{Q}}(\likrat \cdot {\bf 1}_{\tot})}{\mathbb{E}^2_{\mathbb{Q}}(\likrat \cdot {\bf 1}_{\tot})} = 0. 
\]
\end{theorem}

\begin{proof}
By the same argument as in \eqref{eqgeomseriesarg}, we compute, for some positive $c_4$ \rdd{and any real-valued $k \geq 1$},
\[
\E_{\Q} (\likrat^k \cdot {\bf 1}_{\tot \setminus \dom} ) 
= \sum_{r=1+d(\init,g)}^\infty \; \sum_{\pth \in \ipths^r (\init)} \Q(\pth) \left( \frac{\P(\pth)}{\Q(\pth)}\right)^k
\leq \rdb{c_4} \epsilon^{1+k d(\init,\gstate)}
= O(\epsilon^{1+k d(\init,\gstate)})
= \epsilon \cdot O(v^k(\init)).
\]
Furthermore:
\[
\begin{split}
\E_{\Q} (\likrat^k \cdot {\bf 1}_{\dom} ) 
 &= \sum_{\pth \in \dom (\init)} \P(\pth) \left( \frac{\P(\pth)}{\Q(\pth)}\right)^{k-1} = \sum_{\pth \in \dom (\init)} \P(\pth) \left(
      \prod_{i=1}^{\pthl} \frac{\sum_{\stateb \in \ssX} \ppp_{\ps{i-1} \stateb} v(\stateb)}{v(\ps{i})} 
    \right)^{k-1}
 \\
 &= \sum_{\pth \in \dom (\init)} \P(\pth) \left(
      \prod_{i=1}^{\pthl} \frac{v(\ps{i-1})}{v(\ps{i})} 
      (1+O(\epsilon))^{\pthl}
    \right)^{k-1}
 \\
 &= \sum_{\pth \in \dom (\init)} \P(\pth) \left( \frac{v(\init)}{v(g)} \right)^{k-1} \cdot (1+O(\epsilon))
  = v^k(\init) \cdot (1+O(\epsilon))
 .
\end{split}
\]
The second equality uses \eqref{eq ch4 com}, noting that for dominant paths $\pthl =m(\omega)$;
the third equality uses the fact that $v(z)$ is the sum of the dominant paths;
and in the fourth equality \rdb{$(1+O(\epsilon))^{\pthl} = 1+O(\epsilon)$} is justified because $\pthl$ is finite,
as it is bounded from above by the maximum length of a dominant path through the finite set of states $\core$.
Comparing the above two results,
we see that the contribution of the dominant paths dominates for all moments of the estimator.
Hence:
\[
\frac{\vr_{\mathbb{Q}}(\likrat \cdot {\bf 1}_{\tot})}{\mathbb{E}^2_{\mathbb{Q}}(\likrat \cdot {\bf 1}_{\tot})}
= \frac{\E_{\mathbb{Q}}(\likrat^2 \cdot {\bf 1}_{\tot})}{\mathbb{E}^2_{\mathbb{Q}}(\likrat \cdot {\bf 1}_{\tot})} - 1
= O(\epsilon).
\]
Note that we cannot simply invoke Theorem~1 from \citet{lecuyer2011approximating}, because
we have changed the model outside $\core$.
\end{proof}


\begin{theorem}
\label{th:bna}
If $v(x) = \Theta(\epsilon^{\ddd(x, \gstate)})$  \rdb{uniformly in $\state$ and} 
if the estimator based on $v$ and $\Q$ according to~\eqref{eq ch4 com}
does \textbf{not} have vanishing relative error,
\rdb{then} it has the Bounded Normal Approximation (BNA) property:
\[
\frac{\E_\Q(| \likrat \cdot {\bf 1}_{\tot} - \E_\Q (\likrat \cdot {\bf 1}_{\tot}) |^3)}{(\vr_\Q (\likrat \cdot {\bf 1}_{\tot}))^{\frac{3}{2}}} = O(1). 
\]
\end{theorem}

\begin{proof}
Observe that in general for any positive $a$ and $b$,
it holds that
$|a-b|^3 \leq \rdb{(a+b)^3} =  a^3 +3a^2 b +3ab^2 +b^3$.
Applying this to the numerator, we find it is upper-bounded by the sum of four expectation terms, each of which is
of order $O(\epsilon^{3 d(\init, \gstate)})$ by Lemma~\ref{thm: bounded moments},
\rdb{so the same holds for the numerator as a whole}.

For the denominator we find, again using Lemma~\ref{thm: bounded moments}:
\[
\vr_\Q (\likrat \cdot {\bf 1}_{\tot})
=
\E_\Q( \likrat^2 \cdot {\bf 1}_{\tot}) - \E^2_\Q (\likrat \cdot {\bf 1}_{\tot}) 
=
O(\epsilon^{2 d(\init, \gstate)})
\]
If the variance does not vanish (condition of the theorem),
the latter $O$ can be replaced by $\Theta$, completing the proof.
\end{proof}

\begin{corollary}
ZVA-$\dom$ has VRE and ZVA-$\ddd$ has BRE.
\label{th:zva}
\end{corollary}
\begin{proof}
ZVA-$\dom$ uses $v(x)=v^\dom(x)$ from \eqref{eqdefvdom},
which by definition (and by construction in the algorithms of Section~\ref{sec:graphalgorithms})
satisfies the requirement of Theorem~\ref{th:vre}.
Similarly, ZVA-$\ddd$ uses $v(x)=\epsilon^{\ddd(x)}$
which clearly satisfies the requirement of Theorem~\ref{th:bre}.
\end{proof}

%

\section{Variance Reduction For Free?} \label{sec: vrff}

\rda{As part of} the \rdc{Path-ZVA algorithm} \rda{we compute} $\www(\init)$, the probability of the dominant paths from $\init$ to $\gstate$ after HPC removal. When we run the simulation, we implicitly estimate this probability again through the sampling of dominant paths, which affects the estimator variance. Hence, we will explore the possibility of achieving further variance reduction for the estimator $\ipprobest$ by using this by-product of the numerical part of the algorithm. As before, let $\tot = \tot(\init)$, and let $\dom = \dom(\init)$, $\rda{\ndom} = \tot \setminus \dom$, and $\P$ the probability measure after HPC removal. \rda{In words, $\ndom$ is the set of paths that are not dominant but which still contribute to the probability of interest.} We will discuss \rda{two} variations: one in which $\P(\dom)$ is used, and one in which we also compute $\Q(\dom)$.

In the first variation, we use the fact that we already know $\P(\dom)$ by ignoring all runs in which a dominant path is sampled. To see how this is done, note that
 \begin{equation}
\P(\tot)  = \P(\dom) + \P(\ndom)   = \P(\dom) + \mathbb{E}_{\Q}(\rda{\likrat \cdot {\bf 1}_{\ndom}}).
\label{eq vrffplus}
\end{equation}

If we only estimate the \rda{final} expectation in \eqref{eq vrffplus}, we obtain the following estimator:
\begin{equation}
\ipprobest^{\pl} \df \P(\dom) + \frac{1}{N} \sum_{i=1}^N \likrat(\pth_i) \cdot \rda{{\bf 1}_{\ndom}(\pth_i)}.  \label{eq estplus}
\end{equation}
This is equivalent to setting to zero all likelihood ratios obtained from the sampling of dominant paths, and \rdb{adding} $\P(\dom)$ \rdb{to the final result}.

In the second variation,
\rda{we also compute $\Q(\dom)$ by running the same procedure that we used for $\P(\dom)$,
but under the new measure.} 
We then use the fact that
\begin{equation}
\begin{split}
\P(\tot) & = {\P(\dom) + \mathbb{E}_{\Q}(\likrat \cdot \rda{{\bf 1}_{\ndom}})} \\ 
& = \P(\dom) + \mathbb{E}_{\Q}(L_\Q|\rda{\ndom}) \cdot \Q(\ndom).
\end{split} \label{eq vrffplusplus}
\end{equation}
\rda{Although we have not explicitly computed $\Q(\ndom)$, it holds under ZVA that $\Q(\tot) = 1$ because transitions to $\tstate$ are given probability zero. Hence, $\Q(\ndom) = 1 - \Q(\dom)$.}
In practice, we again generate samples $\pth_1,\ldots,\pth_N$,
but if $\pth_i$ is \rda{not in $\ndom$} we \emph{discard} it,
giving rise to the alternative sample $\pth'_1,\ldots,\pth_M'$
where $M$ is the number of samples that are not in~$\dom$.
The resulting estimator is given by:
\begin{equation}
\ipprobest^{\plpl}  \df \P(\dom) + \Q(\ndom)\, Y
\text{ with }
Y = \begin{cases}
  \frac{1}{M} \sum_{i=1}^M \likrat(\pth'_i) & \text{if }M>0  \\
  0                                         & \text{if }M=0.
  \end{cases}
\label{eq estplusplus}
\end{equation}
The separate treatment of $M=0$ is needed to avoid division by zero,
but does not affect the consistency of the estimator.
\rda{Note that we do not need to multiply $ \likrat(\pth'_i)$ by ${\bf 1}_{\tot}(\pth'_i)$ because $\Q(\tot) = 1$.}

Next, let us calculate the variance of $\ipprobest^{\plpl}$:
\newcommand{\Eq}{\E_\Q}
\begin{equation*}
\begin{split}
  \vrq(\ipprobest^{\plpl})
  &= \vrq \left( \P(\dom) + \Q(\ndom) Y \right)
   = \Q(\ndom)^2 \vrq Y
  \\
  &= \Q(\ndom)^2 ( \Eq \vrq(Y|M) + \vrq \Eq(Y|M) )
  \\
  &= \Q(\ndom)^2 \Eq \left. \begin{cases}
      \frac{1}{M} \vrq ( \likrat | \ndom) & \text{if }M>0  \\
      0                                    & \text{if }M=0
    \end{cases} \right\}
 +   \Q(\ndom)^2 \vrq \left. \begin{cases}
      \Eq ( \likrat | \ndom) & \text{if }M>0  \\
      0                                    & \text{if }M=0
    \end{cases} \right\}
\\
  &\approx \Q(\ndom)^2 \frac{\vrq ( \likrat | \ndom)}{N \Q(\ndom)}
  = \frac{\Q(\ndom)}{N} \vrq ( \likrat | \ndom ).
\end{split}
\end{equation*}
where the second line uses the law of total variance,
and the approximation in the fourth line is the limit for $N\rightarrow\infty$.
This limit is motivated by observing that
$M$ has a binomial distribution with parameters $N$ and $\Q(\ndom)$,
which becomes increasingly peaked around its mean $N \Q(\ndom)$ as $N\rightarrow\infty$.

Next, decompose the variance of the original importance sampling estimator $\ipprobest$:
\begin{equation*}
\begin{split}
\vrq(\ipprobest) & = \frac{1}{N} \vrq(\likrat \cdot \i_{\tot})  \\
& = \frac{1}{N} \left(\E_\Q[\vrq(\likrat \cdot \i_{\tot} | \i_{\dom})] + \vrq[\E_\Q(\likrat \cdot \i_{\tot}|\i_\dom)]\right) \\
& = \frac{1}{N} \Q(\ndom)\cdot\vrq(\likrat \cdot \i_{\tot} | \ndom) + \frac{1}{N} \Q(\dom)\cdot\vrq(\likrat \cdot \i_{\tot} | \dom)
+ \frac{1}{N} \vrq[\E_\Q(\likrat \cdot \i_{\tot}|\i_\dom)],
\\
& \approx \vrq(\ipprobest^{++}) + \frac{1}{N} \Q(\dom)\cdot\vrq(\likrat \cdot \i_{\tot} | \dom) + \frac{1}{N} \vrq[\E_\Q(\likrat \cdot \i_{\tot}|\i_\dom)].
\end{split}
\end{equation*}
The latter two terms in this equation are variances and, hence, positive, meaning that \rda{$\ipprobest$} will (for large $N$) have larger variance than $\ipprobest^{++}$.
This will be demonstrated using a case study in Section~\ref{sec:exp-varfree}.

\section{Experimental Results} \label{sec: results}


\newcommand{\figureOne}{
\begin{tikzpicture}[->,>=stealth',shorten >=1pt,auto,node distance=3cm,semithick,scale=0.8]
  \tikzstyle{every state}=[fill=white,draw=black,thick,text=black,minimum size=1cm,scale=0.5]
  \foreach \y in {0,1,2,3} {
		\foreach \x in {0,1,2,3} {
			\FPeval{\xx}{clip(\x*2.5)}%
			\FPeval{\yy}{clip(\y*1.5)}%
			\FPeval{\xy}{clip(\x+\y)}%
			\ifthenelse{\xy = 5}{
				\node[state, fill=yellow]  at (\xx,\yy) (s\x\y) {{\large $\x,\y$}};
			}{
				\ifthenelse{\xy > 5}{
					\node[state, fill=orange!80]  at (\xx,\yy) (s\x\y) {{\large $\x,\y$}};
				}{
					\node[state, fill=white]  at (\xx,\yy) (s\x\y) {{\large $\x,\y$}};
				}
			}
		}
  }
	
	\tikzstyle{every state}=[fill=red!20,draw=black,thick,text=black,minimum size=1cm,scale=0.5]
	
	\foreach \x in {0,1,2,3,4} {
		\FPeval{\xx}{clip(\x*2.5)}%
		\node[state]  at (\xx,6) (s\x4) {{\huge $g$}};
	}
	
	\foreach \y in {0,1,2,3} {
		\FPeval{\yy}{clip(\y*1.5)}%
		\node[state]  at (10,\yy) (s4\y) {{\huge $g$}};
	}
	
	\tikzstyle{every state}=[fill=blue!20,draw=black,thick,text=black,minimum size=1cm,scale=0.5]
	\node[state]  at (-0.75,-0.75) (tstate) {{\huge $t$}};
	
	  \tikzstyle{every state}=[fill=white,draw=black,thick,text=black,scale=0.5]
  \foreach \y in {0,1,2,3} {
		\foreach \x in {0,1,2,3} {
			\FPeval{\xx}{clip(\x+1)}%
			\FPeval{\yy}{clip(\y+1)}%
			\ifthenelse{\x = 0 \AND \y = 0}{
				\path (s\x\y) edge [bend right] node[above] {\footnotesize $\frac{c}{c+1}$} (s\xx\y);
				\path (s\x\y) edge [bend left] node[right] {\footnotesize $\frac{1}{c+1}$} (s\x\yy);
			}{
				\path (s\x\y) edge [bend right] node[below=4pt, left=-10pt] {\footnotesize $\sapprox c \epsilon$} (s\xx\y);
				\path (s\x\y) edge [bend left] node[below=1pt, left=-1pt] {\footnotesize $\sapprox \epsilon$} (s\x\yy);
			}
		}
  }
	
	\foreach \y in {1,2,3} {
		\foreach \x in {1,2,3} {
			\FPeval{\xx}{clip(\x-1)}
			\FPeval{\yy}{clip(\y-1)}
			\path (s\x\y) edge [bend right] node[above=-7pt, right=-7pt] {\footnotesize $\sapprox\frac{1}{2}$} (s\xx\y);
			\path (s\x\y) edge [bend left] node[above=0pt, right=-5pt] {\footnotesize $\sapprox \frac{1}{2}$} (s\x\yy);
		}
  }
	
	\foreach \y in {2,3} {
		\FPeval{\yy}{clip(\y-1)}
		\path (s0\y) edge [bend left] node[above=0pt, right=-5pt] {\footnotesize $\sapprox {1}$} (s0\yy);
	}
	\path (s01) edge [bend right] node[left] {\footnotesize $\sapprox {1}$} (tstate);
	
	\foreach \x in {2,3} {
		\FPeval{\xx}{clip(\x-1)}
		\path (s\x0) edge [bend right] node[above=-5pt, right=-5pt] {\footnotesize $\sapprox {1}$} (s\xx0);
	}
	\path (s10) edge [bend left] node[below] {\footnotesize $\sapprox {1}$} (tstate);
	
\end{tikzpicture}
}


\newcommand{\figureTwo}{
\begin{tikzpicture}[->,>=stealth',shorten >=1pt,auto,node distance=3cm,semithick,scale=0.8]
  \tikzstyle{every state}=[fill=white,draw=black,thick,text=black,minimum size=1cm,scale=0.5]
  \foreach \y in {0,1} {
		\foreach \x in {0,1,2,3,4} {
			\FPeval{\xx}{clip(\x*2.5)}%
			\FPeval{\yy}{clip(\y*1.5)}%
			\FPeval{\xy}{clip(\x+\y)}%
			\ifthenelse{\xy = 3}{
				\node[state, fill=yellow]  at (\xx,\yy) (s\x\y) {{\large $\x,\y$}};
			}{
				\ifthenelse{\xy > 3}{
					\node[state, fill=orange!80]  at (\xx,\yy) (s\x\y) {{\large $\x,\y$}};
				}{
					\node[state, fill=white]  at (\xx,\yy) (s\x\y) {{\large $\x,\y$}};
				}
			}
		}
  }
	
	\tikzstyle{every state}=[fill=red!20,draw=black,thick,text=black,minimum size=1cm,scale=0.5]
	
	\foreach \x in {0,1,2,3,4,5} {
		\FPeval{\xx}{clip(\x*2.5)}%
		\node[state]  at (\xx,3) (s\x2) {{\huge $g$}};
	}
	
	\foreach \y in {0,1} {
		\FPeval{\yy}{clip(\y*1.5)}%
		\node[state]  at (12.5,\yy) (s5\y) {{\huge $g$}};
	}
	
	\tikzstyle{every state}=[fill=blue!20,draw=black,thick,text=black,minimum size=1cm,scale=0.5]
	\node[state]  at (-0.75,-0.75) (tstate) {{\huge $t$}};
	
	  \tikzstyle{every state}=[fill=white,draw=black,thick,text=black,scale=0.5]
  \foreach \y in {0,1} {
		\foreach \x in {0,1,2,3,4} {
			\FPeval{\xx}{clip(\x+1)}%
			\FPeval{\yy}{clip(\y+1)}%
			\ifthenelse{\x = 0 \AND \y = 0}{
				\path (s\x\y) edge [bend right] node[above] {\footnotesize $\frac{c}{c+1}$} (s\xx\y);
				\path (s\x\y) edge [bend left] node[right] {\footnotesize $\frac{1}{c+1}$} (s\x\yy);
			}{
				\path (s\x\y) edge [bend right] node[below=4pt, left=-10pt] {\footnotesize $\sapprox c\epsilon$} (s\xx\y);
				\path (s\x\y) edge [bend left] node[below=1pt, left=-1pt] {\footnotesize $\sapprox \epsilon$} (s\x\yy);
			}
		}
  }
	
	\foreach \y in {1} {
		\foreach \x in {1,2,3,4} {
			\FPeval{\xx}{clip(\x-1)}
			\FPeval{\yy}{clip(\y-1)}
			\path (s\x\y) edge [bend right] node[above=-7pt, right=-7pt] {\footnotesize $\sapprox\frac{1}{2}$} (s\xx\y);
			\path (s\x\y) edge [bend left] node[above=0pt, right=-5pt] {\footnotesize $\sapprox \frac{1}{2}$} (s\x\yy);
		}
  }
	
	\path (s01) edge [bend right] node[left] {\footnotesize $\sapprox {1}$} (tstate);
	
	\foreach \x in {2,3,4} {
		\FPeval{\xx}{clip(\x-1)}
		\path (s\x0) edge [bend right] node[above=-5pt, right=-5pt] {\footnotesize $\sapprox {1}$} (s\xx0);
	}
	\path (s10) edge [bend left] node[below] {\footnotesize $\sapprox {1}$} (tstate);
	
\end{tikzpicture}
}


\newcommand{\figureHpcDefGrp}{
\begin{tikzpicture}[->,>=stealth',shorten >=1pt,auto,node distance=3cm,semithick,scale=0.8]
  \tikzstyle{every state}=[fill=white,draw=black,thick,text=black,minimum size=1cm,scale=0.5]
  \foreach \y in {0,1} {
		\foreach \x in {0,1,2,3,4} {
			\FPeval{\xx}{clip(\x*2.5)}%
			\FPeval{\yy}{clip(\y*1.5)}%
			\FPeval{\xy}{clip(\x+\y)}%
			\ifthenelse{\xy = 4}{
				\node[state, fill=yellow]  at (\xx,\yy) (s\x\y) {{\large $\x,\y$}};
			}{
				\ifthenelse{\xy > 4}{
					\node[state, fill=orange!80]  at (\xx,\yy) (s\x\y) {{\large $\x,\y$}};
				}{
					\node[state, fill=white]  at (\xx,\yy) (s\x\y) {{\large $\x,\y$}};
				}
			}
		}
  }
	
	\tikzstyle{every state}=[fill=red!20,draw=black,thick,text=black,minimum size=1cm,scale=0.5]
	
	\foreach \x in {0,1,2,3,4,5} {
		\FPeval{\xx}{clip(\x*2.5)}%
			\node[state]  at (\xx,4.5) (s\x2) {{\large $\x, 2$}};
	}
	
	\foreach \y in {0, 1} {
		\FPeval{\yy}{clip(\y*1.5)}%
		\node[state]  at (12.5,\yy) (s5\y) {{\large $5, \y$}};
	}
	
	\tikzstyle{every state}=[fill=blue!20,draw=black,thick,text=black,minimum size=1cm,scale=0.5]
	\node[state]  at (-0.75,-0.75) (tstate) {{\huge $t$}};
	
	  \tikzstyle{every state}=[fill=white,draw=black,thick,text=black,scale=0.5]
  \foreach \y in {0,1} {
		\foreach \x in {0,1,2,3,4} {
			\FPeval{\xx}{clip(\x+1)}%
			\FPeval{\yy}{clip(\y+1)}%
			\ifthenelse{\x = 0 \AND \y = 0}{
				\path (s\x\y) edge [] node[below=6.5pt, left=-8pt] {\footnotesize $\frac{c}{c+1}$} (s\xx\y);
				\path (s\x\y) edge [] node[right] {\footnotesize $\frac{1}{c+1}$} (s\x\yy);
			}{
				\ifthenelse{\x = 1 \AND \y = 0}{
					\path (s\x\y) edge [] node[below=6.5pt, left=-8pt] {\footnotesize $\frac{c}{c+1}$} (s\xx\y);
					\path (s\x\y) edge [bend left] node[left] {\footnotesize $\frac{1}{c+1}$} (s\x\yy);
				}{
					\ifthenelse{\y = 0}{
						\path (s\x\y) edge [] node[below=4pt, left=-10pt] {\footnotesize $\sapprox c\epsilon$} (s\xx\y);
						\path (s\x\y) edge [bend left] node[below=1pt, left=-1pt] {\footnotesize $\sapprox \epsilon$} (s\x\yy);
					} {
						\path (s\x\y) edge [] node[below=4pt, left=-10pt] {\footnotesize $\sapprox c\epsilon$} (s\xx\y);
						\path (s\x\y) edge [bend left] node[below=-15pt, left=-1pt] {\footnotesize $\sapprox \epsilon$} (s\x\yy);
					}
				}
			}
		}
  }
	
		\foreach \x in {2,3,4} {
			\path (s\x1) edge [bend right] node[above=2pt, right=22pt] {\footnotesize $\sapprox\frac{1}{2}$} (s01);
			\path (s\x1) edge [bend left] node[above=0pt, right=-5pt] {\footnotesize $\sapprox \frac{1}{2}$} (s\x0);
		}
		\path (s11) edge [bend left] node[above=0pt, right=-5pt] {\footnotesize $\sapprox 1$} (s10);
	
	\path (s01) edge [bend right] node[left] {\footnotesize $\sapprox {1}$} (tstate);
	
	\foreach \x in {2,3,4} {
		\path (s\x0) edge [bend left] node[above=-5pt, right=-5pt] {\footnotesize $\sapprox {1}$} (tstate);
	}
	
\end{tikzpicture}
}


\newcommand{\figBasicHpc}{
\begin{tikzpicture}[->,>=stealth',shorten >=1pt,auto,node distance=3cm,semithick,scale=0.8] 
		\tikzstyle{every state}=[fill=blue!20,draw=black,thick,text=black,scale=0.5]
		\node at (0,0) (s0) {};
		\node[state] at (0,-0.6) (s02) {\huge $\tstate$};
		
		\tikzstyle{every state}=[fill=white,draw=black,thick,text=black,scale=0.5]
		\node[state] at (0,0.6) (s01) {\huge $\init$};
		\node[state] (s1) [right of = s0] {\huge $1$};
		\node[state] (s2) [right of = s1] {\huge $2$};
		\node[state] (s3) [right of = s2] {\huge $3$};
		\node[state] (s4) [right of = s3] {\huge $4$};
		
		\tikzstyle{every state}=[fill=red!20,draw=black,thick,text=black,scale=0.5]
		\node[state] (sr) [right of = s4] {\huge $\gstate$};

			\path (s01) edge [bend left] node[above] {\footnotesize $1$} (s1);
			\path (s1) edge  [bend left] node[below] {\footnotesize \textcolor{black}{$\epsilon$}} (s02);
			\path (s1) edge  [bend left] node[above] {\footnotesize \textcolor{black}{$1-\epsilon$}} (s2);
			\path (s2) edge  [bend left] node[below] {\footnotesize $1-\epsilon$} (s1);
			\path (s2) edge  [bend left] node[above] {\footnotesize $\epsilon$} (s3);
			\path (s3) edge  [bend left] node[below] {\footnotesize $1-\epsilon$} (s2);
			\path (s3) edge  [bend left] node[above] {\footnotesize $\epsilon$} (s4);
			\path (s4) edge  [bend left] node[below] {\footnotesize $1-\epsilon$} (s3);
			\path (s4) edge  [] node[above] {\footnotesize $\epsilon$} (sr);
			
	\end{tikzpicture}
}

\newcommand{\figBasicHpcRmvd}{
\begin{tikzpicture}[->,>=stealth',shorten >=1pt,auto,node distance=3cm,semithick, scale=0.8]					
		\tikzstyle{every state}=[fill=blue!20,draw=black,thick,text=black,scale=0.5]
		\node at (0,0) (s0) {};
		\node[state] at (0,-0.6) (s02) {\huge $\tstate$};
		
		\tikzstyle{every state}=[fill=white,draw=black,thick,text=black,scale=0.5]
		\node[state] at (0,0.6) (s01) {\huge $\init$};
		\node[state] (s1) [right of = s0] {\huge $1$};
		\node[state] (s2) [right of = s1] {\huge $2$};
		\node[state] (s3) [right of = s2] {\huge $3$};
		\node[state] (s4) [right of = s3] {\huge $4$};
		
		\tikzstyle{every state}=[fill=red!20,draw=black,thick,text=black,scale=0.5]
		\node[state] (sr) [right of = s4] {\huge $\gstate$};

			\path (s01) edge [bend left] node[above] {\footnotesize $1$} (s1);
			\path (s1) edge  [bend left] node[above] {\footnotesize \textcolor{black}{$\frac{1}{2-\epsilon}$}} (s02);
			\path (s1) edge  [bend left] node[above] {\footnotesize \textcolor{black}{$\frac{1-\epsilon}{2-\epsilon}$}} (s3);
			\path (s2) edge  [bend left] node[below] {\footnotesize \textcolor{black}{$\frac{1-\epsilon}{2-\epsilon}$}} (s02);
			\path (s2) edge  [bend left] node[below] {\footnotesize \textcolor{black}{$\frac{1}{2-\epsilon}$}} (s3);
			\path (s3) edge  [bend left] node[below] {\footnotesize $1-\epsilon$} (s2);
			\path (s3) edge  [bend left] node[above] {\footnotesize $\epsilon$} (s4);
			\path (s4) edge  [bend left] node[below] {\footnotesize $1-\epsilon$} (s3);
			\path (s4) edge  [] node[above] {\footnotesize $\epsilon$} (sr);
	\end{tikzpicture}
}

In this section, we present the results of simulation experiments with the Path-ZVA method.
The aim of the experiments is twofold. In Sections~\ref{sec: toy examples}, we focus on illustrative examples meant to demonstrate theoretical results and elucidate core concepts,
namely the BRE and VRE properties (Sec.\ \ref{sec:exp-bre-vre}),
the nature of $\core$ and $\edge$ in a practical example (Sec.\ \ref{sec:exp-core-edge}),
HPC removal (Sec.\ \ref{sec:exp-hpc}),
and the performance of $\ipprobest^{\pl}$ and $\ipprobest^{\plpl}$ (Sec.\ \ref{sec:exp-varfree}).
In Section~\ref{sec: realistic examples} we demonstrate the good performance of the new method using several realistic models from the literature. We compare it to the BFB and IGBS methods discussed in Section~\ref{sec: related work}, and to the results for two case studies presented by \citet{carrasco2006failure}.
All of the experiments were conducted using a general framework written in Java, and the code needed to run the experiments is available on \website. All experiments involve a particular class of models, namely highly reliable \emph{multicomponent systems}. Although this is already a very broad class of models, we emphasise that our procedure works for any HRMS (see \rdc{the sample models included with the algorithm's code} for several other applications, such as a 2-node tandem queue).

The simulation methods that we consider are: standard Monte Carlo (MC), BFB and IGBS from the literature,
and two variations of our Zero Variance Approximation method,
namely ZVA-$\ddd$ and ZVA-$\dom$
as defined in Section~\ref{sec:pathzva}.

\newcommand{\sapprox}{\approx\hspace{-2.5pt}}
\vspace{-0.3cm}
\begin{figure}[!ht]
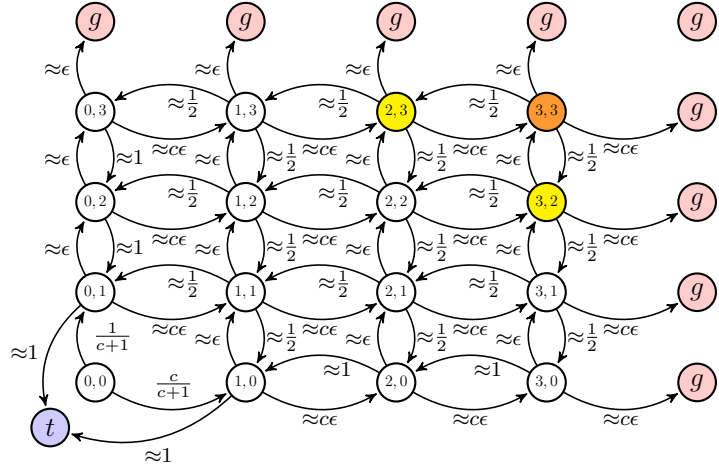

\caption{Example of the case study of Section~\ref{sec: basic mcs}, with $k_1 = k_2 = 4$.
The blue state is~$t$, the pink states (marked~$g$) are to be merged into a single state~$g$,
the yellow states (2,3) and (3,2) form $\edge$,
the states in $\core \setminus \{\gstate, \tstate\}$ are white,
and ${\cal X} \setminus ( \Lambda \cup \Gamma)$ in this case contains only state (3,3), coloured orange. 
}

\centering
\figureOne
\label{fig: first case study}
\end{figure}

\subsection{Illustrative Examples} \label{sec: toy examples}

\subsubsection{A Basic Example} \label{sec: basic mcs}
\label{sec:exp-bre-vre}
\label{sec:exp-core-edge}

Our first example is a multicomponent system with two component types, and $k_1$ and $k_2$ components of types~1~and~2 respectively. \rdd{The system states are denoted by $(x_1,x_2)$, in which $x_i, i \in\{1, 2\},$ is the number of components of type $i$ that have failed.} For each component type, one component is active at each time, with the other components acting as spares. The rate at which the active component of type 1 fails equals $c\epsilon$, $c \in (0, \infty)$, while the active component of type 2 fails with rate $\epsilon$. Each component type has a dedicated repair unit which begins work immediately after the first component has failed, and which repairs a single component with a rate of $1$. The system as a whole fails if all components of at least one of the two types have failed. 
\rdd{Both the initial state $s$ and regeneration state $t$ are (0,0);
	as usual, we are interested in the probability
	of reaching a failure state $g$ before returning to (0,0).}

\rda{A DTMC is created for this model (and all other models in this section) by assigning to transitions from $\statea$ to $\stateb$, with $\statea, \stateb \in \ssX$, a probability equal to the rate of transitions from $\statea$ to $\stateb$ divided by the total exit rate of state $\statea$.} A graphical representation of such a DTMC is given in Figure~\ref{fig: first case study} for $k_1=k_2=4$. The model has no HPCs, and, depending on $k_1$ and $k_2$, the dominant paths are given by the two straight paths from $(0,0)$ to $(k_1,0)$ and $(0,k_2)$. If $k_1 = k_2$, both paths are dominant, otherwise the shortest path is the unique dominant path. It holds that $d(\init, \gstate) = \min(k_1,k_2)$, and a state $(x_1,x_2)$ is in~$\core$ iff $x_1 + x_2 \leq \min(k_1,k_2)$.

\setlength{\tabcolsep}{0.75pt}
\begin{table}[t]
 \centering
\tbl{Confidence intervals (95\%) for $\ipf$ as functions of $\epsilon$ for the different simulation methods, for the model of Figure~\ref{fig: first case study}. Sample size: 10\,000 runs.}{
\begin{tabular}{|ccc|clcccclcccclcccclccc|}
	\hline
	\rule{0pt}{1.1em}
	\phantom{-} & $\epsilon$ & \phantom{-} & \phantom{-} & \multicolumn{3}{c}{MC} & \phantom{-} & \phantom{-} & \multicolumn{3}{c}{BFB} & \phantom{-} & \phantom{-} & \multicolumn{3}{c}{ZVA-$\ddd$} & \phantom{-} & \phantom{-} & \multicolumn{3}{c}{ZVA-$\dom$} & \phantom{-} \\ \hline
	 & 0.1 &  & & $1.038$\ex{-3} & $\pm$ & 6.08\% & & & $9.955$\ex{-4} & $\pm$ & 1.00\% & & & $9.994$\ex{-4} & $\pm$ & 0.12\% & & & $9.999$\ex{-4} & $\pm$ & 0.10\% &\\ 
	 & 0.01 &  & &  \multicolumn{3}{c}{---}  & & & $1.010$\ex{-6} & $\pm$ & 1.46\% & & & $1.000$\ex{-6} & $\pm$ & 0.04\% & & & $1.000$\ex{-6} & $\pm$ & 0.03\% &\\ 
	 & 0.001 &  & &  \multicolumn{3}{c}{---}  & & & $9.940$\ex{-10} & $\pm$ & 1.55\% & & & $1.000$\ex{-9} & $\pm$ & 0.01\% & & & $1.000$\ex{-9} & $\pm$ & 0.01\% &\\ 
	 & 1.0E-4 &  & &  \multicolumn{3}{c}{---}  & & & $1.014$\ex{-12} & $\pm$ & 1.54\% & & & $1.000$\ex{-12} & $\pm$ & 0.00\% & & & $1.000$\ex{-12} & $\pm$ & 0.00\% &\\ 
	\hline
\end{tabular}}
\label{tab: method relative errors}
\end{table}
\setlength{\tabcolsep}{6pt}

In Table~\ref{tab: method relative errors}, we present a summary of a basic simulation experiment with different values of $\epsilon$ for each of the main simulation methods discussed in this paper, performed on the model with $k_1=k_2=4$. It can be seen that ZVA does much better than the other methods for sufficiently small values of $\epsilon$.
We expect VRE for ZVA-$\dom$ and BRE for ZVA-$\ddd$ by Corollary~\ref{th:zva},
which is indeed confirmed by the table.

\begin{figure}
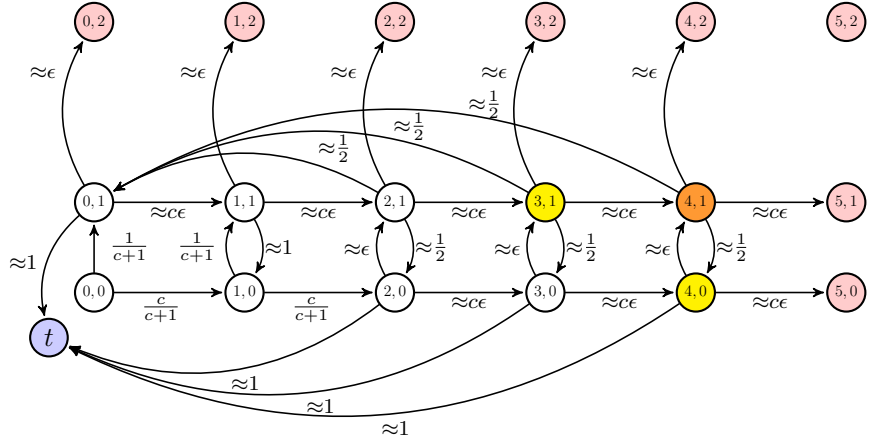

\caption{Same system as in Figure~\ref{fig: first case study}, except with $k_1=5$ and $k_2=2$. Also, component type~1 is subject to deferred group repair, meaning that repair starts when 2 components have failed, and all components are repaired at the same time. \rda{The colouring is the same as in Figure~\ref{fig: first case study}. }}
\centering
\figureHpcDefGrp
\label{fig: basic hpc example}
\end{figure}

\subsubsection{Group/Deferred Repair}
\label{sec:exp-hpc}

We now discuss the impact of HPCs on the performance of the various \rda{importance sampling} methods. HPCs can emerge naturally in a multicomponent system if repair strategies are used that cause repairs to be slow or inactive in certain states of the system. It is known that BFB does not do well when HPCs are present; to remedy this, a more intricate version of BFB has been proposed, called IGBS \cite{juneja2001fast}; see Section~\ref{sec: related work} for more details. 
In this section, we will see that BFB will not do well in this setting, and IGBS only in some cases depending on the choice of parameters.

\setlength{\tabcolsep}{0.75pt}
\begin{table}[!t]
 \centering
\tbl{Confidence intervals (95\%) for $\ipf$ as a function of $\epsilon$ for the model of Figure~\ref{fig: basic hpc example}, with $c = \frac{1}{50}$. Here, $\pi \approx 100 / 51 \cdot \epsilon \approx 1.96 \epsilon$. Sample size: 10\,000 runs.
A `---' means that the rare event was not observed at all.
}{
\begin{tabular}{|ccc|clcccclcccclcccclcccclccc|}
	\hline
	\rule{0pt}{1.1em}
	\phantom{-} & $\epsilon$ & \phantom{-} & \phantom{-} & \multicolumn{3}{c}{MC} & \phantom{-} & \phantom{-} & \multicolumn{3}{c}{BFB} & \phantom{-} & \phantom{-} & \multicolumn{3}{c}{IGBS} & \phantom{-} & \phantom{-} & \multicolumn{3}{c}{ZVA-$\ddd$} & \phantom{-} & \phantom{-} & \multicolumn{3}{c}{ZVA-$\dom$} & \phantom{-} \\ \hline
	 & 0.1 &  & & $1.133$\ex{-1} & $\pm$ & 5.48\% & & & $9.865$\ex{-2} & $\pm$ & 6.63\% & & & $8.933$\ex{-2} & $\pm$ & 26.5\% & & & $1.065$\ex{-1} & $\pm$ & 1.92\% & & & $1.063$\ex{-1} & $\pm$ & 0.01\% &\\ 
	 & 0.001 &  & & $1.300$\ex{-3} & $\pm$ & 54.3\% & & & $1.166$\ex{-3} & $\pm$ & 19.9\% & & & $2.241$\ex{-3} & $\pm$ & 20.5\% & & & $1.892$\ex{-3} & $\pm$ & 6.61\% & & & $1.912$\ex{-3} & $\pm$ & 0.00\% &\\ 
	 & 1.0E-5 &  & &  \multicolumn{3}{c}{---}  & & & $1.180$\ex{-5} & $\pm$ & 19.7\% & & & $1.718$\ex{-5} & $\pm$ & 23.9\% & & & $2.038$\ex{-5} & $\pm$ & 6.78\% & & & $1.960$\ex{-5} & $\pm$ & 0.00\% &\\ 
	 & 1.0E-7 &  & &  \multicolumn{3}{c}{---}  & & & $1.140$\ex{-7} & $\pm$ & 6.96\% & & & $1.921$\ex{-7} & $\pm$ & 23.5\% & & & $2.019$\ex{-7} & $\pm$ & 6.79\% & & & $1.961$\ex{-7} & $\pm$ & 0.00\% &\\ 
	\hline
\end{tabular}}
\label{tab: basic hpc}
\end{table}
\setlength{\tabcolsep}{6pt}

The setting that we consider first is depicted as a DTMC in Figure~\ref{fig: basic hpc example}. Here, $k_1 = 5$ and $k_2 = 2$, and the repair strategy for component type 1 includes both deferred and group repair. Deferred repair means that the repair unit for component type 1 will not begin work until a minimum number of components have broken down --- two in this case. Group repair means that when repair has begun, all components are repaired at the same time. The DTMC contains an HPC between states $(1,0)$ and $(1,1)$. This has a large impact on the dominant paths. Specifically, one dominant path is the path $((0,0), (0,1), (0,2))$, which occurs with probability $\frac{1}{c+1} \epsilon$. The other dominant paths are those that jump from $(0,0)$ to $(1,0)$, then cycle between $(1,0)$ and $(1,1)$ $k$ times, $k \in \{0,1,\ldots\}$, and then jump to $(1,2)$. These paths have a total probability contribution of
\[
  \frac{c}{(c+1)^{\rda{2}}} \cdot \sum_{i=0}^{\infty} \left(\frac{1}{c+1} \cdot \frac{1}{1+(c+1)\epsilon} \right)^i \cdot \epsilon \; = \; \rda{\frac{1}{c+1}}\epsilon + o(\epsilon).
\]
so for small $\epsilon$ roughly one half of the total probability mass is contributed by the path going to $(0,2)$ and the other half by the ones going to $(1,2)$. 

During the pre-processing step, the HPC is detected by Algorithm~\ref{alg: forward phase} when the transition from state $(1,1)$ to state $(1,0)$ is considered. At that point, state $(1,0)$ has already been determined to be in $\core$, whilst the `cost' of reaching these states in terms of $\epsilon$-orders is the same. Hence, the condition in line~\ref{ln: loop detect if} is satisfied, which triggers the HPC removal procedure of Algorithm~\ref{alg: loop detect}. Note that for the model of Figure~\ref{fig: first case study}, a HPC is (correctly) not detected because the `cost' to reach states $(1,0)$ and $(1,1)$ is different. The set of states in the HPC in Figure~\ref{fig: basic hpc example} (i.e., the set $L$ of line~\ref{ln: def L}) equals $\{(1,0), (1,1)\}$. This means that all transitions within $L$ are removed and the probabilities of ending up in states $(2,0)$, $(2,1)$, and $(1.2)$ from the two states in $L$ are determined via line~\ref{ln: scc merge}. For example, for $\epsilon = \frac{1}{100}$, this leads to probabilities of roughly $66\%$, $0.6\%$, and $33\%$ of reaching states $(2,0)$, $(2,1)$ and $(1,2)$ respectively from state $(1,0)$.

BFB will not do well for small values of $c$; a cycle occurs with a probability of roughly $1/(c+1)$ under $\P$, which is close to one if $c$ is close to zero, but BFB will only assign probability $\frac14$ to these cycles. This means that the dominant paths that contain many cycles will be sampled infrequently, resulting either in underestimation (see \citet{devetsikiotis1993algorithmic}) when these paths are not sampled in a simulation experiment, or high relative errors if they are sampled as each cycle blows up the likelihood ratio roughly by a factor $4/(c+1) \approx 4$. IGBS mitigates the impact of this phenomenon by setting the probability of each HPC to $\delta^2$ instead of $\frac14$, with $\delta^2 < \frac14$. Still, `good' choices of $\delta$ depend on $c$, so this requires a non-trivial knowledge of the system. This is illustrated in Table~\ref{tab: basic hpc}, in which BFB can be seen to suffer from underestimation (as witnessed by, e.g., its confidence interval not containing the true value of approximately $1.960\cdot 10^{-5}$ for $\epsilon = 10^{-5}$). \rdd{Note that the confidence interval bounds in the first columns do not seem
	trustworthy, probably because we have too few samples and/or very large fourth
	moments.} By contrast IGBS (with $\delta = \frac{1}{100}$) is accurate in the sense that its confidence interval contains the true value, although it does not perform as well as ZVA.

\setlength{\tabcolsep}{0.75pt}
\begin{table}[t]
 \centering
\tbl{Confidence intervals (95\%) for $\ipf$ as a function of $\epsilon$ for the model of Figure~\ref{fig: basic hpc example} with two changes: $(\rdd{k_1, k_2}) = (5, 3)$, and for components of type 1, the first fails with rate $\epsilon$ and the others with rate $\epsilon^2$.
Sample size: 10\,000 runs.
Note that ZVA-$\rda{\ddd}$ has VRE because of the specific simple structure of $\www$ in this model.
A confidence interval width of `---' means that no variance was observed.
}{
\begin{tabular}{|ccc|clcccclcccclcccclccc|}
	\hline
	\rule{0pt}{1.1em}
	\phantom{-} & $\epsilon$ & \phantom{-} & \phantom{-} & \multicolumn{3}{c}{BFB} & \phantom{-} & \phantom{-} & \multicolumn{3}{c}{IGBS} & \phantom{-} & \phantom{-} & \multicolumn{3}{c}{ZVA-$\ddd$} & \phantom{-} & \phantom{-} & \multicolumn{3}{c}{ZVA-$\dom$} & \phantom{-} \\ \hline
	 & 0.1 &  & & $2.692$\ex{-2} & $\pm$ & 46.0\% & & & $9.391$\ex{-2} & $\pm$ & 128\% & & & $4.651$\ex{-2} & $\pm$ & 1.30\% & & & $4.644$\ex{-2} & $\pm$ & 1.02\% &\\ 
	 & 0.01 &  & & $3.440$\ex{-4} & $\pm$ & 47.0\% & & & $6.125$\ex{-3} & $\pm$ & 62.4\% & & & $4.939$\ex{-3} & $\pm$ & 0.44\% & & & $4.961$\ex{-3} & $\pm$ & 0.33\% &\\ 
	 & 0.001 &  & & $2.933$\ex{-6} & $\pm$ & 57.1\% & & & $1.348$\ex{-4} & $\pm$ & 29.2\% & & & $4.995$\ex{-4} & $\pm$ & 0.14\% & & & $4.987$\ex{-4} & $\pm$ & 0.13\% &\\ 
	 & 1.0E-4 &  & & $6.664$\ex{-8} & $\pm$ & 99.9\% & & & $1.797$\ex{-6} & $\pm$ & 41.2\% & & & $4.998$\ex{-5} & $\pm$ & 0.05\% & & & $5.000$\ex{-5} & $\pm$ & 0.03\% &\\ 
	 & 1.0E-5 &  & & $4.093$\ex{-10} & $\pm$ & 57.9\% & & & $3.654$\ex{-8} & $\pm$ & 79.9\% & & & $4.999$\ex{-6} & $\pm$ & 0.02\% & & & $5.000$\ex{-6} & $\pm$ & --- &\\ 
	\hline
\end{tabular}}
\label{tab: extreme hpc}
\end{table}
\setlength{\tabcolsep}{6pt}

As the value of $c$, and therefore the probability of leaving the HPC, is decreased, the performance of IGBS will worsen. In the extreme case where the probability of leaving the HPC decreases proportionally with $\epsilon$, this is particularly visible. Consider the following modifications to the previous example: $k_1$ now equals $3$, and the failure rate for components of type 1 is $\epsilon$ for the first component and $\epsilon^2$ for the spare components. In Table~\ref{tab: extreme hpc}, we have displayed the results for this setting. Here, IGBS does not contain the true value of approximately $\frac12\epsilon$ for smaller values of $\epsilon$. We have not included standard MC because of the very large amount of time it takes to sample runs.

\subsubsection{Variance reduction for free} 
\label{sec:exp-varfree}

Table~\ref{tab: method relative errors plus} shows a comparison of the different estimators discussed in Section~\ref{sec: vrff}, using the model of Figure~\ref{fig: basic hpc example}.
We see that $\ipprobest^{\plpl}$ has notably better performance than the standard estimator, whereas $\ipprobest^{\pl}$ is worse.
The difference between $\ipprobest^{\plpl}$ and $\ipprobest$ varies between models --- e.g., for the model of Figure~\ref{fig: first case study} their performance is roughly the same, and in some models we have even observed $\ipprobest^{\plpl}$ performing worse than $\ipprobest^{\pl}$.
However, as is evident from Table~\ref{tab: method relative errors plus}, the potentially minor cost of performing the numerical pre-processing step a second time can lead to a reduction in confidence interval width of over 75\% (e.g., see the row for $\epsilon = 0.01$).
We will consider the pre-processing runtimes in more detail in the next section. Note that when no non-dominant paths are drawn (i.e., $M=0$), $\ipprobest^{\plpl}$ is no longer able to produce an estimate of the estimator variance and $\ipprobest$ is to be preferred.
\setlength{\tabcolsep}{0.75pt}
\begin{table}[!ht]
 \centering
\tbl{Confidence intervals (95\%) for $\ipf$ as functions of $\epsilon$ for the different simulation methods \rdd{discussed in Section~\ref{sec: vrff}}, for the model of Figure~\ref{fig: basic hpc example}.}{
\begin{tabular}{|ccc|clcccclcccclccc|ccccccccc|}
	\hline
	\rule{0pt}{1.1em}
	\phantom{-} & $\epsilon$ & \phantom{-} & \phantom{-} & \multicolumn{3}{c}{\rdd{$\ipprobest$}} & \phantom{-} & \phantom{-} & \multicolumn{3}{c}{\rdd{$\ipprobest^{\pl}$}} & \phantom{-} & \phantom{-} & \multicolumn{3}{c}{\rdd{$\ipprobest^{\pl\pl}$}} & \phantom{-} & \phantom{-} & $N$ & \phantom{-} & \phantom{-} & $M$ & \phantom{-} & \phantom{-} & $N \Q(\dom)$ & \phantom{-}  \\ \hline
	 & 0.1 &  & & $1.063$\ex{-1} & $\pm$ & 0.0092\% & &  & $1.063$\ex{-1} & $\pm$ & 0.2327\% & &  & $1.063$\ex{-1} & $\pm$ & 0.0018\% &  & & $10\,000$ & & & $141$ & & & $143$ & \\ 
	 & 0.01 &  & & $1.622$\ex{-2} & $\pm$ & 0.0060\% & &  & $1.621$\ex{-2} & $\pm$ & 0.1161\% & &  & $1.622$\ex{-2} & $\pm$ & 0.0001\% &  & & $10\,000$ & & & $35$ & & & $42$ & \\ 
	 & 0.001 &  & & $1.912$\ex{-3} & $\pm$ & 0.0009\% & &  & $1.912$\ex{-3} & $\pm$ & 0.0392\% & &  & $1.912$\ex{-3} & $\pm$ & --- &  & & $10\,000$ & & & $4$ & & & $5$ & \\ 
	 & 1.0E-4 &  & & $1.956$\ex{-4} & $\pm$ & 0.0001\% & &  & $1.956$\ex{-4} & $\pm$ & --- & &  & $1.9556$\ex{-4} & $\pm$ & ---  &  & & $10\,000$ & & & $0$ & & & $1$ & \\ 
	\hline
\end{tabular}}
\label{tab: method relative errors plus}
\end{table}
\setlength{\tabcolsep}{6pt}

\subsection{Realistic Examples} \label{sec: realistic examples}

In this section we demonstrate the good performance of the Path-ZVA approach using two models from the literature. \rdd{The first is} the Distributed Database System, a classic literature benchmark that has been studied since the seventies \cite{rosenkrantz1978system}, but which remains relevant today. 
In Section~\ref{sec: dds}, we study the variation from \citet{boudali2008arcade} and \citet{reijsbergen2010rare}, and use Path-ZVA to compare the performance of different repair strategies. In Section~\ref{sec: fcs}, we study the variation from \citet{carrasco2006failure}, and the fault-tolerant computing system \rdd{of the same paper}.

Instead of $\ipf$, the probability of reaching the goal set during a \rdd{regeneration} cycle (i.e., before returning to the taboo state), the probability of interest in \cite{carrasco2006failure} is the system \emph{unavailability}, denoted here by $\unaiv$. It is defined as the steady-state probability of being in the goal set. \rdd{W}e will also consider this measure in this section \rdd{in order to compare results}.
We use $\unaiv=\E(Z)/\E(D)$,
where $D$ is the total duration of a \rdd{regeneration} cycle (i.e., time between two visits to the taboo state) and $Z$ is the amount of time spent in the goal set during a \rdd{regeneration} cycle. 
Typically we estimate $\E(D)$ using standard MC,
while $\E(Z)$ is estimated based on an estimate of $\ipf$.
For a more elaborate discussion, see, e.g., \citet{reijsbergen2010rare}. Note that HPC removal does have non-trivial consequences for state sojourn times and hence estimates for $\unaiv$, although this does not affect the case studies because they do not have HPCs.

Additionally, each table now also displays the run times and \rdd{Work-Normalized Variance Ratios (WNVRs)} with respect to standard MC. We use the following \rdd{WNVR} definition: For a method $m$, let $w_m$ be its confidence interval half-width and $\rho_m$ the total time needed to produce the result. Then the \rdd{WNVR} for this method is given by $(w_{\text{MC}} / w_{m})^2 \cdot \rho_{\text{MC}} / \rho_{m}$. The \rdd{WNVR} represents the fact that to reduce the confidence interval width by a factor $c$ one would need to draw $c^2$ as many samples. It allows for easy comparison between methods with different runtimes; higher values of the \rda{WNVR} indicate better performance. 

\subsubsection{The Distributed Database System} \label{sec: dds}

In this variant, the system consists of 9 component types: one set of 2 processors, two sets of 2 controllers each, and 6 disk clusters, with 6 disks each; see 
Figure~\ref{fig: dds}. The failure rates for individual components are $\epsilon^2 / 2$ for processors and disk controllers, and $\epsilon^2 / 6$ for disks. The rates of component repairs are 1 for processors and disk controllers, and $\epsilon$ for disks. \rdc{Note that the $\epsilon$-orders are not part of the benchmark setting: the disk repairs being asymptotically slower than the other repairs is specific to this paper. An interpretation would be that the data on the disks needs to be replicated whereas the processors and the disk controllers only require hardware replacement.} If we had assigned the same $\epsilon$-order to the repairs of each of the types then the four repair strategies would have the exact same asymptotic performance.
The total failure rate for each component type depends linearly on the number of working components of that type; e.g., four working disks in disk set 1 means a total failure rate of $\rdd{2\epsilon^2/3}$ for disk set 1.
\rdc{The system as a whole is down if both processors are down, if both disk controllers in one of the controller sets are down, or if four disks are down in a single cluster.} \rdd{Both $\init$ and $\tstate$ are the state where all components are up.} We consider four repair strategies:
\begin{enumerate}
\item A dedicated repair unit for each of the 9 component types. \label{it: ded rep}
\item One repair unit, with priority given to high component type indices (i.e., disks first, then controllers, then processors).
\item One repair unit, with priority given to low component type indices (i.e., processors first, then controllers, then disks).
\item One repair unit, with a First Come First Served (FCFS) policy. \label{it: fcfs}
\end{enumerate}

From a modelling point of view, Strategy~\ref{it: fcfs} is the least tractable; to keep track of the order in which the components failed, a vector representing the number of failed components of each type is not sufficient. Specifically, if $k$ components are down, then there are $k!$ ways in which this could have happened chronologically. This poses two problems. First\rdd{,} the size of the state space blows up dramatically, from 421\,875 to 2\,123\,047\,371 states. Second, if a modelling language is used that does not support lists (e.g., PRISM's reactive modules language), even a high-level description of the model can be hard to give. However, in the Java framework that we use for the experiments, states that contain lists are not conceptually harder to implement than vectors. \rdc{The sizes of the sets $\core$ and $\edge$ for the four strategies are as follows: 155 and 399 for dedicated repair, 561 and 448 for disk priority, 175 and 463 for processor priority, and 578 and 4\,428 for FCFS. In all cases, $\core \cup \edge$ is much smaller than the full state space.}

\begin{figure}
	\caption{Distributed Database System.}
	\label{fig: dds}
  \begin{center}
    \includegraphics[width=0.42\textwidth]{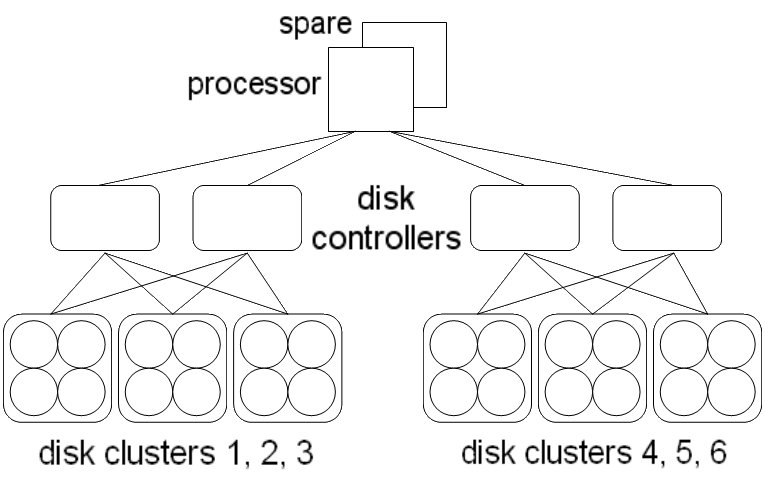}
  \end{center}
\end{figure}

In Table~\ref{tab: dds slow disks strategies}, we compare the four repair strategies in terms of their performance. Disk priority and FCFS are much more failure prone than the other strategies, because system failure due to two processors or disk controllers breaking becomes more likely if the repair unit is working on a disk. Apart from ZVA-$\dom$, we also present
results obtained using the model checking tool PRISM, which approximates the probability
of interest using numerical techniques (e.g., Gauss-Seidel) applied to the transition
probability matrix. We see that our methods are accurate, albeit less efficient than
PRISM, which was typically able to find the probability of interest within a second.

\setlength{\tabcolsep}{0.75pt}
\begin{table}[!ht]
 \centering
\tbl{Confidence intervals (95\%) as a function of $\epsilon$, generated using Path-ZVA; comparison of repair strategies for the DDS with slow disk repairs.}{
\begin{tabular}{|cccccc|clccc|cccccccccccc|}
	\hline
	\rule{0pt}{1.1em}
	\phantom{-} & \multirow{2}{*}{model} & \phantom{-} & \phantom{-} & \multirow{2}{*}{$\epsilon$} & \phantom{-} &  \phantom{-} & \multicolumn{3}{c}{\multirow{2}{*}{estimate ($\ipf$)}} & \phantom{-} & \phantom{-} & \multirow{2}{*}{$N$} & \phantom{-} & \phantom{-} & \multicolumn{4}{c}{runtime (ms)} & \phantom{-} & \phantom{-} & \multirow{2}{*}{PRISM} & \phantom{-} \phantom{-}  \\
	 & & & & & &  & & & & & & & & & num. & & & sim. & & & &  \\ \hline
	 & \multirow{4}{*}{DDS, ded. rep.} & &  & \multirow{1}{*}{0.1} & &  & $2.997$\ex{-3} & $\pm$ & 6.81\% & & & 69272 & & & 1663 & & & 10000 & & & $3.441$\ex{-3} &  \\
	 & & &  & \multirow{1}{*}{0.03} & &  & $1.802$\ex{-4} & $\pm$ & 1.71\% & & & 815322 & & & 1501 & & & 10000 & & & $1.859$\ex{-4} &  \\
	 & & &  & \multirow{1}{*}{0.01} & &  & $1.786$\ex{-5} & $\pm$ & 1.43\% & & & 1522273 & & & 1420 & & & 10000 & & & $1.790$\ex{-5} &  \\
	 & & &  & \multirow{1}{*}{0.003} & &  & $1.543$\ex{-6} & $\pm$ & 1.71\% & & & 1495413 & & & 1440 & & & 10000 & & & $1.532$\ex{-6} &  \\ \hline
	 & \multirow{4}{*}{DDS, disk prior.} & &  & \multirow{1}{*}{0.1} & &  & $4.466$\ex{-2} & $\pm$ & 16.6\% & & & 4361 & & & 1117 & & & 10003 & & & $4.925$\ex{-2} &  \\
	 & & &  & \multirow{1}{*}{0.03} & &  & $1.634$\ex{-3} & $\pm$ & 6.92\% & & & 14205 & & & 1041 & & & 10000 & & & $1.704$\ex{-3} &  \\
	 & & &  & \multirow{1}{*}{0.01} & &  & $1.356$\ex{-4} & $\pm$ & 2.54\% & & & 33536 & & & 1042 & & & 10002 & & & $1.367$\ex{-4} &  \\
	 & & &  & \multirow{1}{*}{0.003} & &  & $1.099$\ex{-5} & $\pm$ & 0.66\% & & & 103258 & & & 1056 & & & 10000 & & & $1.101$\ex{-5} &  \\ \hline
	 & \multirow{4}{*}{DDS, proc. prior.} & &  & \multirow{1}{*}{0.1} & &  & $8.642$\ex{-3} & $\pm$ & 13.4\% & & & 3292 & & & 790 & & & 10007 & & & $8.419$\ex{-3} &  \\
	 & & &  & \multirow{1}{*}{0.03} & &  & $1.944$\ex{-4} & $\pm$ & 6.18\% & & & 19236 & & & 783 & & & 10000 & & & $1.954$\ex{-4} &  \\
	 & & &  & \multirow{1}{*}{0.01} & &  & $1.797$\ex{-5} & $\pm$ & 2.51\% & & & 54301 & & & 787 & & & 10001 & & & $1.798$\ex{-5} &  \\
	 & & &  & \multirow{1}{*}{0.003} & &  & $1.520$\ex{-6} & $\pm$ & 0.87\% & & & 161383 & & & 789 & & & 10000 & & & $1.533$\ex{-6} &  \\ \hline
	 & \multirow{4}{*}{DDS, FCFS} & &  & \multirow{1}{*}{0.1} & &  & $3.058$\ex{-2} & $\pm$ & 22.1\% & & & 1278 & & & 38116 & & & 10007 & & & --- &  \\
	 & & &  & \multirow{1}{*}{0.03} & &  & $1.384$\ex{-3} & $\pm$ & 16.9\% & & & 3512 & & & 40400 & & & 10004 & & & --- &  \\
	 & & &  & \multirow{1}{*}{0.01} & &  & $1.304$\ex{-4} & $\pm$ & 6.14\% & & & 5068 & & & 41954 & & & 10002 & & & --- &  \\
	 & & &  & \multirow{1}{*}{0.003} & &  & $1.066$\ex{-5} & $\pm$ & 2.50\% & & & 7483 & & & 38669 & & & 10000 & & & --- &  \\
	\hline
\end{tabular}}
\label{tab: dds slow disks strategies}
\end{table}
\setlength{\tabcolsep}{6pt}

\subsubsection{Fault-Tolerant Control/Database Systems} \label{sec: fcs}

Two models are presented by \citet[Section~VI]{carrasco2006failure}: the Fault-Tolerant Database System (FTD) and the Fault-Tolerant Control System (FTC). The FTD is a variation of the Distributed Database System \rdd{discussed in the previous section} --- \rdd{the goal and taboo sets are the same}. The FTD has 10 component types; however, there are two types of failures \rdd{so} the state is represented using a 20-dimensional vector. Additionally, the model has failure propagation: a failure of a processor of the first type may trigger a failure of a processor of the second type. There are two parameter settings (I and II). In setting I the system is `balanced' in the sense that the $\epsilon$-orders of all failure transitions equals 1, whereas in setting II some failures have $\epsilon$-order 1 and others $\epsilon$-order 2. The second model, the FTC, consists of 39 component types, and system failure is a non-trivial function of the state. \rdd{Because of space constraints, we refer the reader to \citet{carrasco1992failure} or our programming code for a full description of the model.} We only consider the first out of four possible parameter settings for the FTC. The technique proposed in the paper, called Balanced Failure Transition Distance Biasing (BFTDB), is a refinement of the method proposed by \citet{carrasco1992failure} to ensure good performance for unbalanced systems.

As we can see from Table~\ref{tab: carrasco comparison}, Path-ZVA has roughly similar performance to BFTDB, which is to be expected since they are based on the same principles. BFTDB does slightly better than Path-ZVA for the FTC because of the relatively large probability contribution of paths that leave $\core$ --- the numerical procedure behind BFTDB determines $\ddd$ for all states, which means that it is able to perform better in this specific setting. (Note that their numerical approach cannot be applied to general HRMSs, for example those that include HPCs). BFB does not perform well in our implementation because it draws much fewer samples per second than the other schemes. This is because we use the default biasing probability of $0.5$ for failures, which means that a typical sample path will be considerably longer than under the other schemes. For example, under MC the sample path will typically reach the taboo state very quickly, whereas under Path-ZVA the system quickly reaches the goal state or a state outside $\core$ after which IS is turned off.

Note that in all models the transition rates are fixed, so the choice of $\epsilon$ is arbitrary. As we discussed in Section~\ref{sec: practical application}, our approach is to fix a value $\epsilon$ and choose the $\epsilon$-orders of the transitions as the smallest order such that the pre-factor $p_{\state\stateb} / \epsilon^{\asyrate_{\state\stateb}}$ is still greater than $\epsilon$. The $\epsilon$-values chosen by \citet{carrasco2006failure} were $0.00072$ for both settings of the FTD and $0.00028$ for setting A of the FTC. We have observed that using an $\epsilon$-value of $0.001$ for the FTD led to a large reduction in terms of the size of $\core\cup\edge$ and hence the duration of the pre-processing step, without adversely affecting the performance of Path-ZVA to a notable extent. This is what we have used for Table~\ref{tab: carrasco comparison}.

\setlength{\tabcolsep}{0.75pt}
\begin{table}[!ht]
 \centering
\tbl{Comparison with three reliability models from the literature; BFTDB is the simulation method proposed in [Carrasco 2006]. For the FTD we used $\epsilon=0.001$, and for the FTC we used $\epsilon=0.00072$. We use a confidence level of 95\%, whereas [Carrasco 2006] uses 99\%; the results have been rescaled accordingly.
Runtimes for BFTDB are from [Carrasco 2006], while the other runtimes are from our tool;
because of software implementation and hardware differences, the comparison is at best indicative. \rdd{For our experiments, we chose the simulation runtime in each case to be around $2$ minutes, leading to different numbers of runs for the different methods.}}{
\begin{tabular}{|cccccc|clccc|ccccccccc|ccc|}
	\hline
	\rule{0pt}{1.1em}
	\phantom{-} & \multirow{2}{*}{model} & \phantom{-} & \phantom{-} & \multirow{2}{*}{method} & \phantom{-} &  \phantom{-} & \multicolumn{3}{c}{\multirow{2}{*}{estimate ($\unaiv$)}} & \phantom{-} & \phantom{-} & \multirow{2}{*}{$N$} & \phantom{-} & \phantom{-} & \multicolumn{4}{c}{runtime (ms)} & \phantom{-} & \phantom{-} & \multirow{2}{*}{\rdd{WNVR}} & \phantom{-} \\
	 & & & & & &  & & & & & & & & & num. & & & sim. & & & & \\ \hline
	 & \multirow{4}{*}{FTD (I)} & &  & MC & &  & $1.827$\ex{-8} & $\pm$ & 8.10\% & & & 32184109 & & & 0 & & & 120294 & & & 1.00 & \\
	 & & &  & BFB & &  & $1.931$\ex{-8} & $\pm$ & 14.9\% & & & 27396 & & & 0 & & & 120026 & & & 0.30 & \\
	 & & &  & ZVA-$\ddd$ & &  & $1.827$\ex{-8} & $\pm$ & 0.23\% & & & 2554961 & & & 5994 & & & 120005 & & & 1178.93 & \\
	 & & &  & ZVA-$\dom$ & &  & $1.829$\ex{-8} & $\pm$ & 0.19\% & & & 2726637 & & & 4755 & & & 120001 & & & 1847.67 & \\
	 & & &  & BFTDB & &  & $1.814$\ex{-8} & $\pm$ & 0.23\% & & & 2999000 & & & 0 & & & 112000 & & & 268.22 &  \\ \hline
	 & \multirow{4}{*}{FTD (II)} & &  & MC & &  & $1.728$\ex{-8} & $\pm$ & 6.51\% & & & 42956445 & & & 0 & & & 120001 & & & 1.00 & \\ 
	 & & &  & BFB & &  & $1.636$\ex{-8} & $\pm$ & 12.5\% & & & 28372 & & & 0 & & & 120005 & & & 0.27 & \\
	 & & &  & ZVA-$\ddd$ & &  & $1.622$\ex{-8} & $\pm$ & 0.20\% & & & 3110348 & & & 522 & & & 120001 & & & 1011.29 & \\ 
	 & & &  & ZVA-$\dom$ & &  & $1.621$\ex{-8} & $\pm$ & 0.22\% & & & 2756305 & & & 527 & & & 120001 & & & 888.40 & \\
	 & & &  & BFTDB & &  & $1.621$\ex{-8} & $\pm$ & 0.15\% & & & 3999000 & & & 0 & & & 150000 & & & 284.54 &  \\ \hline
	 & \multirow{4}{*}{FTC (A)} & &  & MC & &  & $3.232$\ex{-10} & $\pm$ & 50.7\% & & & 20827971 & & & 0 & & & 120001 & & & 1.00 & \\
	 & & &  & BFB & &  & $1.908$\ex{-10} & $\pm$ & 77.8\% & & & 3008 & & & 0 & & & 120052 & & & 0.43 & \\
	 & & &  & ZVA-$\ddd$ & &  & $2.682$\ex{-10} & $\pm$ & 0.94\% & & & 827848 & & & 448936 & & & 120048 & & & 609.53 & \\
	 & & &  & ZVA-$\dom$ & &  & $2.696$\ex{-10} & $\pm$ & 0.38\% & & & 937337 & & & 478349 & & & 120054 & & & 3627.10 & \\ 
	 & & &  & BFTDB & &  & $2.694$\ex{-10} & $\pm$ & 0.15\% & & & 1204000 & & & 0 & & & 319000 & & & 8112.75 &  \\ \hline
\end{tabular}}
\label{tab: carrasco comparison}
\end{table}
\setlength{\tabcolsep}{6pt}

\section{Conclusions} \label{sec: mc conclusion}

We have introduced a rare event simulation method that is generally 
applicable to HRMSs, provided that the relevant subset $\Lambda$ is numerically tractable.
\rdf{This is often the case, but not always, e.g., when the reliability of the system is due to high component redundancy}.
We have mathematically 
proved its efficiency and discussed an automated implementation. We have demonstrated its good performance across a range of case studies, including a realistic benchmark model. For one repair strategy (First Come First Served), the new method was able to compute probabilities that cannot be obtained using either standard Monte Carlo or the numerical approximation techniques used in, e.g., PRISM. We also discussed a further variance reduction technique and demonstrated its good performance. The code \rdd{for} the experiments is available on \website\ for download.

There are several directions for future work. The simulation code has not been optimised for performance, so improving it is future work. The variance reduction technique of Section~\ref{sec: vrff} could be studie\rdb{d} in more detail, and across a wider range of models. Finally, we could compare the performanc\rda{e o}f our method to a wider range of other IS techniques, e.g., the cross-entropy method.

\begin{acks}
{This work is partially supported by the Netherlands Organisation for Scientific Research (NWO), project number 612.064.812, and by the EU projects QUANTICOL, 600708, and SENSATION, 318490. \rdc{The authors would like to thank Jane Hillston for her helpful comments on a draft version of this paper}.}
\end{acks}

\bibliographystyle{ACM-Reference-Format-Journals}
\bibliography{biblio}

\end{document}